\documentclass[12pt,oneside,reqno]{amsart}
\usepackage[all]{xy}
\usepackage{amsfonts,amsmath,oldgerm,amssymb,amscd, comment,multirow}
\UseComputerModernTips
\numberwithin{equation}{section}
\usepackage[breaklinks]{hyperref}
\allowdisplaybreaks

\usepackage{enumerate}

\usepackage{caption} 
\newtheorem{theorem}{Theorem}[section]

\newtheorem{lemma}[theorem]{{\bf Lemma}}
\newtheorem{coro}[theorem]{{\bf Corollary}}
\newtheorem{definition}[theorem]{Definition}

\newtheorem{remark}[subsection]{Remark}

\newcommand{\al}{\alpha}

\begin{document}

	\title[ Distribution and congruences of $(u,v)$-regular bipartitions ]{ Distribution and congruences of $(u,v)$-regular bipartitions} %\\ \today}
	
	\author[Nabin Kumar Meher*]{Nabin Kumar Meher}
	\address{Nabin Kumar Meher, Department of Mathematics, Indian Institute of Information Raichur, Govt. Engineering College Campus, Yermarus, Raichur, Karnataka, India 584135.}
	\email{mehernabin@gmail.com, nabinmeher@iiitr.ac.in}
	
	%	\author[ Ankita Jindal]{ Ankita Jindal}
	%	\address{Ankita Jindal, Indian Statistical Institute Bangalore, 8th Mile, Mysore Road, Bangalore, Karnataka, India 560059}
	%	\email{ankitajindal1203@gmail.com }

	\maketitle
	\pagenumbering{arabic}
	\pagestyle{headings}
	\begin{abstract}
		Let $B_{u,v}(n)$ denote the number of $(u,v)$-regular bipartitions of $n$. In this article, we prove that $B_{p,m}(n)$ is always almost divisible by $p,$ where $p\geq 5$ is a prime number and $m=p_1^{\al_1} p_2^{\al_2}\cdots p_r^{\al_r}, $ where $\al_i \geq 0$ and $p_i \geq 5$ be distinct primes with $\gcd(p,m)=1$ . Further, we obtain an infinities families of congruences modulo $3$ for $B_{3,7}(n),$ $B_{3,5}(n)$ and $B_{3,2}(n)$ by using Hecke eigenform theory and a result of Newman \cite{Newmann1959}. Furthermore, we get many infinite families of congruences modulo $7$, $11$ and $13$  respectively for $B_{2,7}(n)$, $B_{2,11}(n)$  and $B_{2,13}(n),$ by employing an identity of Newman \cite{Newmann1959}. In addition, we prove infinite families of congruences modulo $2$ for $B_{4,3}(n)$, $B_{8,3}(n)$  and $B_{4,5}(n)$ by applying another result of Newman \cite{Newmann1962}.
		\end{abstract}
	\maketitle
	
	\section{Introduction}
	Let $t$ and $n$ be positive integers. A partition of $n$ is called $t$-regular if none of its parts is divisible by $t.$ Let $b_t(n)$ denotes the number of $t$-regular partitions of $n.$ The generating function for $b_{t}(n)$ is given by 
	$$ \sum_{n=0}^{\infty} b_{t}(n) q^n= \frac{f_{t}}{f_1},$$
	where $f_{t}$ is defined by $f_{t}= \prod_{m=1}^{\infty} (1- q^{t m}).$
	Gordan and Ono\cite{Gordan1997} proved a density result on $t$-regular partitions. In particular, they proved that if $p$ is a prime number and $p^{\hbox{ord}_p(t)} \geq \sqrt{t},$ then for any positive integer $j,$ the arithmetic density of positive integers $n$ such that $b_{t}(n) \equiv 0 \pmod{p^j}$ is one.

	Note that a bipartition $(\alpha, \beta)$ of $n$ is a pair of partitions $(\alpha, \beta)$ such that the sum of all of the parts is $n.$   For any positive integers $u \geq 2, v \geq 2,$ an $(u,v)$-regular bipartition of $n$ is a bipartition $(\alpha, \beta)$ of $n$ such that $\alpha$ is a $u$-regular partition and $\beta $ is a $v$-regular partition. Let $B_{u,v}(n)$ denote the number of $(u,v)$-regular bipartitions of $n$. Conventionally, assume that $B_{u,v}(0)=1.$ The generating function of $B_{u,v}(n)$ is given by
		\begin{align}\label{eq501}
		\sum_{n=0}^{\infty} B_{u,v}(n) q^n= \frac{f_{u} f_{v}}{f_1^2}.
	\end{align}
In order to state the following results, we introduce the Legendre symbol. Let $p \geq 3$ be a prime and $a$ be an integer. The Legendre symbol $\left( \frac{a}{p}  \right)_{L}$ is defined as
$$ \left( \frac{a}{p}  \right)_{L}:= \begin{cases} 
	1,  \quad \hbox{if}  \ a \hbox{ is a quadratic residue modulo} p \ \hbox{and}  \ p \nmid a, \\
	-1,  \quad \hbox{if} \ a \hbox{ is a nonquadratic residue modulo} \ p,  \\
	0,  \quad  \hbox{if}  \ p|a. 
\end{cases}$$
	In 2017, Xia and Yao \cite{Xia2017} studied $(k,\ell)$-regular bipartitions and proved that for a positive integer $s$ and a prime $p\geq 5$ with $\left( \frac{-s}{p}\right)= -1,$ we have
	\begin{equation}\label{eq001}
		B_{3,s} \left( p^{2\alpha +1}n+ \frac{(1+s)(p^{2\alpha+2}-1)}{24}\right) \equiv 0 \pmod3
	\end{equation}
	whenever $n, \alpha \geq 0$ with $p \nmid n$. 
Motivated by the above results, we show that $B_{p,m}(n)$ is almost always divisible by a prime number $p\geq 5.$ In particular, we prove the following.

\begin{theorem}\label{mainthm16}  Let $p \geq 5$ be a prime number. Let $m=p_1^{\al_1} p_2^{\al_2}\cdots p_r^{\al_r}, $ where $\al_i \geq 0$  be non-negative integers and $p_i \geq 5$ be distinct primes with $\gcd(p,m)=1$ . Then the set 
	$$\left\{ n \in \mathbb{N} : B_{p,m}(n) \equiv 0 \pmod{p} \right\}$$ has arithmetic density $1.$
\end{theorem}
In 2016, Dou \cite{Dou2016} found that, for $n\geq 0$ and $\alpha \geq 2,$
\begin{equation}\label{eq007}
	B_{3,11} \left( 3^{\alpha}n+ \frac{5 \cdot 3^{\alpha-1}-1}{2}\right) \equiv 0 \pmod{11}.
\end{equation}
 Further, in 2017, L. Wang \cite{Wang2017} studied the arithmetic properties of $B_{3,\ell}(n)$ and $B_{5,\ell}(n).$ In particular, he proved that for any non-negative integer $n,$ we have
\begin{equation}\label{eq005}
	B_{3,7} \left( 2^{2\alpha}n+ \frac{5 \cdot 2^{2\alpha-1}-1}{3}\right) \equiv 0 \pmod3 \quad \hbox{ for all } \alpha \geq 1,
\end{equation}
\begin{equation}\label{eq006}
	B_{3,7}\left(4n+1\right) \equiv - B_{3,7}\left(n\right) \pmod3,
\end{equation}
and
\begin{equation}\label{eq006a}
	B_{3,7}\left( 4n+3 \right) \equiv 0 \pmod3.
\end{equation}
In 2019, T. Kathiravan and K. Srilakshmi \cite{Kathiravan2019} proved infinite families of congruences modulo $5$ for $B_{2,15}(n),$ modulo $11$ for $B_{7,11}(n)$ and modulo $17$ for $B_{243,17}(n).$ Motivated from the above result, by using Hecke eigenform and modular form theory, we obtain the following congruences modulo $3$ for $B_{3,7}(n), B_{3,5}(n) $ and $B_{3,2}(n).$ In particular, it says the following.
 
 \begin{theorem}\label{thm8}
 Let $k$ and $n$ be non-negative integers. Then we have
 	\begin{itemize}
 		\item[(i)] $	B_{3,7} \left( p_1^2 p_2^2 \cdots p_k^2 p_{k+1}^2 n + \frac{p_1^2 p_2^2 \cdots p_k^2 p_{k+1}\left(3j+ p_{k+1}\right)-1}{3} \right) \equiv 0 \pmod3,$  where $p_i$ is a prime number such that $p_i \not \equiv 1 \pmod3$ for $i \in \{ 1, 2, \ldots, k+1\},$ and $ p_{k+1} \nmid j.$
 	\item[(ii)]	$B_{3,5} \left( p_1^2 p_2^2 \cdots p_k^2 p_{k+1}^2 n + \frac{p_1^2 p_2^2 \cdots p_k^2 p_{k+1}\left(4j+ p_{k+1}\right)-1}{4} \right) \equiv 0 \pmod3,$  where $p_i$ is a prime number such that $p_i \equiv 3 \pmod4$ for $i \in \{ 1, 2, \ldots, k+1\},$ and $ p_{k+1} \nmid j.$
 	
 	\item[(iii)]	$B_{3,2} \left( p_1^2 p_2^2 \cdots p_k^2 p_{k+1}^2 n + \frac{p_1^2 p_2^2 \cdots p_k^2 p_{k+1}\left(8j+ p_{k+1}\right)-1}{8} \right) \equiv 0 \pmod 3,$  where $p_i$ is a prime number such that $p_i \not \equiv 1 \pmod8$ for $i \in \{ 1, 2, \ldots, k+1\},$ and $ p_{k+1} \nmid j.$
	\end{itemize}
 \end{theorem}
	If we substitute $p_1= p_2= \cdots= p_{k+1}= p$ in Theorem \ref{thm8}, then we obtain the  following result.
	
	\begin{coro}\label{coro3}
		Let $k$ and $n$ be non-negative integers. Then we have
		\begin{itemize}
			\item[(i)]	$B_{3,7}\left(p^{2k+2}n +p^{2k+1}j+ \frac{p^{2k+2}-1}{3} \right) \equiv 0 \pmod 3,$ where $p$ is an odd prime such that $p \equiv 2 \pmod3,$ and $p \nmid j.$
			\item[(ii)]$B_{3,5}\left(p^{2k+2}n +p^{2k+1}j+ \frac{p^{2k+2}-1}{4} \right) \equiv 0 \pmod 3,$ where $p$ is an odd prime such that $p \equiv 3 \pmod4,$ and $p \nmid j.$
			\item[(iii)] $B_{3,2}\left(p^{2k+2}n +p^{2k+1}j+ \frac{p^{2k+2}-1}{8} \right) \equiv 0 \pmod 3,$ where $p$ is an odd prime such that $p \not \equiv 1 \pmod8,$ and $p \nmid j.$
		\end{itemize}
	\end{coro}
\begin{remark} 
	\begin{itemize}
		\item[(i)] If we substitute $p=2, k=0$ and $j=1$ in \ref{coro3}$(i)$, we obtain \eqref{eq006a}.
		\item[(ii)] If we put $p=2, k=\alpha-1$ and $j=1$ in \ref{coro3}$(i)$, we obtain \eqref{eq005}.
	\end{itemize}
In Particular, our results generalize some of the result obtained by Wang \cite{Wang2017}.
\end{remark}
Further, by employing a result of Newman \cite{Newmann1959}, we obtain another infinite families of congruences modulo $3$ for $B_{3,7}(n), B_{3,5}(n) $ and $B_{3,2}(n).$
\begin{theorem}\label{thmNewmannapplication1} Let $\alpha$ be a non-negative integer.
	\begin{itemize}
		\item[(i)] Let $p$ be a prime number with $p \equiv 1 \pmod 6.$ If $p \nmid (3n+1)$ and $B_{3,7} \left(\frac{p-1}{3}\right)\equiv 0 \pmod 3,$ then for $n \geq 0, $ 
		\begin{equation}\label{eq49}
			B_{3,7}\left(p^{2 \alpha+1}n+ \frac{p^{2 \alpha +1}-1}{3}\right) \equiv 0 \pmod 3.
		\end{equation}
	\item[(ii)] Let $p$ be a prime number with $p \equiv 1 \pmod 4.$ If $p \nmid (4n+1)$ and $B_{3,5} \left(\frac{p-1}{4}\right)\equiv 0 \pmod 3,$ then for $n \geq 0, $ and 
	\begin{equation}\label{eq149}
		B_{3,5}\left(p^{2 \alpha+1}n+ \frac{p^{2 \alpha +1}-1}{4}\right) \equiv 0 \pmod 3.
	\end{equation}
	\item[(iii)] Let $p$ be a prime number with $p \equiv 1 \pmod 8.$ If $p \nmid (8n+1)$ and $B_{3,2} \left(\frac{p-1}{8}\right)\equiv 0 \pmod 3,$ then for $n \geq 0,$
	\begin{equation}\label{eq249}
		B_{3,2}\left(p^{2 \alpha+1}n+ \frac{p^{2 \alpha +1}-1}{8}\right) \equiv 0 \pmod 3.
	\end{equation}
	\end{itemize}
\end{theorem}
Furthermore, we prove the following multiplicative formulae for $B_{3,7}(n),$ $B_{3,5}(n)$ and $B_{3,2}(n)$ modulo $3$.
\begin{theorem}\label{thm9}
	 Let $k$ and $n$ be positive integers.
	\begin{itemize}
		\item[(i)]Let $p$ be a prime number such that $ p \equiv 2 \pmod 3.$ Let $r$ be a non-negative integer such that $p$ divides $3r+2,$ then
		\begin{align*}
			B_{3,7}\left(p^{k+1}n+pr+ \frac{2p-1}{3}\right)\equiv   -\left(\frac{-7}{p}\right)_{L}  B_{3,7}\left(p^{k-1}n+ \frac{3r+2-p}{3p} \right) \pmod3.
		\end{align*}
		\item[(ii)] Let $p$ be a prime number such that $ p \equiv 3 \pmod 4.$ Let $r$ be a non-negative integer such that $p$ divides $4r+3,$ then
		\begin{align*}
			B_{3,5}\left(p^{k+1}n+pr+ \frac{3p-1}{4}\right)\equiv -\left(\frac{-5}{p}\right)_{L}  B_{3,5}\left(p^{k-1}n+ \frac{4r+3-p}{4p} \right) \pmod3.
		\end{align*}
		\item[(iii)] Let $t \in \{3,5,7\}.$	Let $p$ be a prime number such that $ p \equiv t \pmod 8.$ Let $r$ be a non-negative integer such that $p$ divides $8r+t,$ then
		\begin{align*}
			B_{3,2}\left(p^{k+1}n+pr+ \frac{tp-1}{8}\right)\equiv  -\left(\frac{-2}{p}\right)_{L}   B_{3,2}\left(p^{k-1}n+ \frac{8r+t-p}{8p} \right) \pmod3.
		\end{align*}
	\end{itemize}
\end{theorem}

	\begin{coro}\label{coro4}
	Let $k$ and $n$ be positive integers. 
		\begin{itemize}
			\item[(i)] Let $p$ be a prime number such that $p \equiv 2 \pmod 3.$ Then
			\begin{align*}
				B_{3,7}\left(p^{2k}n+ \frac{p^{2k}-1}{3} \right)&\equiv  \left(-1\right)^k \left(\frac{-7}{p}\right)_{L}^k   B_{3,7}(n) \pmod3.
			\end{align*}
		\item[(ii)] 	Let $p$ be a prime number such that $p \equiv 3 \pmod 4.$ Then
		\begin{align*}
			B_{3,5}\left(p^{2k}n+ \frac{p^{2k}-1}{4} \right)&\equiv \left(-1\right)^k \left(\frac{-5}{p}\right)_{L}^k   B_{3,5}(n) \pmod3.
		\end{align*}
		\item[(iii)]  Let $t \in \{3,5,7\}.$ Let $p$ be a prime number such that $p \equiv t \pmod 8.$ Then
		\begin{align*}
			B_{3,2}\left(p^{2k}n+ \frac{p^{2k}-1}{4} \right)&\equiv  \left(-1\right)^k \left(\frac{-2}{p}\right)_{L}^k  B_{3,2}(n) \pmod3.
		\end{align*}
		\end{itemize}
\end{coro}
\begin{remark}
	In particular, if we put $p=2, k=1$ in the above corollary \ref{coro4}$(i)$, we get \eqref{eq006}.
\end{remark}
Next, by applying an identity due to Newman \cite{Newmann1959}, we obtain the following congruences modulo $7,11$ and $13$ respectively.

\begin{theorem}\label{newmannthm4}
	\begin{itemize}
		\item[(i)] Let $p$ be a prime number with $p \equiv 1 \pmod {24}.$ If $p \nmid (24n+7)$ and $B_{7,2} \left(\frac{7(p-1)}{24}\right)\equiv 0 \pmod 7,$ then for $n \geq 0, $ and $\alpha \geq 0,$
		\begin{equation}\label{eq349}
			B_{7,2}\left(p^{2 \alpha +1}n+ \frac{7(p^{2 \alpha+1}-1)}{24}\right) \equiv 0 \pmod 7.
		\end{equation}
	\item[(ii)] 	Let $p$ be a prime number with $p \equiv 1 \pmod {24}.$ If $p \nmid (24n+1)$ and $B_{11,2} \left(\frac{11(p-1)}{24}\right)\equiv 0 \pmod{11},$ then for $n \geq 0, $ and $\alpha \geq 0,$
	\begin{equation}\label{eq449}
		B_{11,2}\left(p^{2 \alpha +1}n+ \frac{11(p^{2 \alpha+1}-1)}{24}\right) \equiv 0 \pmod{11}.
	\end{equation}
	\item[(iii)] Let $p$ be a prime number with $p \equiv 1 \pmod {24}.$ If $p \nmid (24n+13)$ and $B_{13,2} \left(\frac{13(p-1)}{24}\right)\equiv 0 \pmod{13},$ then for $n \geq 0, $ and $\alpha \geq 0,$
	\begin{equation}\label{eq549}
		B_{13,2}\left(p^{2 \alpha +1}n+ \frac{13(p^{2 \alpha+1}-1)}{24}\right) \equiv 0 \pmod{13}.
	\end{equation}
	\end{itemize}
\end{theorem}
In addition, by employing another result of Newman \cite{Newmann1962}, we obtain the following congruences of modulo $2$ for $B_{4,3}(n).$
\begin{theorem}\label{thmnewman10}
	Let $a_1(n)$ be defined by
	\begin{equation}\label{eq1001}
		\sum_{n=0}^{\infty} a_1(n)q^n:= \left(q;q\right)^2_{\infty} \left(q^3;q^3\right)_{\infty} 
	\end{equation}
	and let $p \geq 5$ be a prime number.
	Define 
	\begin{equation}\label{eq1002}
		w_1(p):= a_1\left( \frac{5(p^2-1)}{24}\right) + \left( \frac{-6}{p}\right)_{L} \left( \frac{-5(p^2-1)}{24 p}\right)_{L}
	\end{equation}
	\begin{itemize}
		\item[(i)] If $ w_1(p) \equiv 0 \pmod2,$ then for $n,k \geq 0,$ if $p \nmid n,$ we have
		\begin{equation}\label{eq1003}
			B_{4,3} \left( p^{4k+3}n+  \frac{5 \left( p^{4k+4}-1\right)}{24} \right) \equiv 0 \pmod2.
		\end{equation}
		\item[(ii)] If $ w_1(p) \not \equiv 0 \pmod2,$ then for $n,k \geq 0,$ if $p \nmid n,$ we have
		\begin{equation}\label{eq1004}
			B_{4,3} \left( p^{6k+5}n+  \frac{5 \left( p^{6k+6}-1\right)}{24} \right) \equiv 0 \pmod2.
		\end{equation}
		
		\item[(iii)] If $ w_1(p) \not \equiv 0 \pmod2,$ then for $n,k \geq 0,$ with $ w_1(p)\equiv  \left( \frac{ -6n-1+ \frac{(p^2-1)}{4}}{p}\right)_{L}\pmod2,$ we have
		\begin{equation}\label{eq1005}
			B_{4,3} \left( p^{6k+2}n+  \frac{5 \left( p^{6k+2}-1\right)}{24} \right) \equiv 0 \pmod2.
		\end{equation}
	\end{itemize}
\end{theorem}
\begin{remark}
	\begin{itemize}
		\item[(i)] Note that \eqref{eq1001} was considered in [Thereom $1.5$, \cite{Xia2021}] and [Theorem $1.1$, \cite{Kathiravan2023}]. They derived 
		 infinite families of congruences modulo $3$ and $2$ respectively.
		 \item[(ii)] If we put $p=7,$ \eqref{eq1002} we obtain that $a_1(10)=0$ and $w_1(7)=1.$ Note that $ w_1(7) \equiv \left( \frac{  11-6n}{7}\right)_{L} \equiv \left( \frac{n-3}{7}\right)_{L} \equiv 1 \pmod2$ when $n \not \equiv 3 \pmod7.$ Thus putting $k=0,$ in \eqref{eq005}, we get 
		 	\begin{equation*}
		 	B_{4,3} \left( 343n+ 49 r+10 \right) \equiv 0 \pmod2,
		 \end{equation*} where $r \in \{ 0, 1, 2, 4, 5, 6\}$
	 	\end{itemize}
\end{remark}
Similarly, by employing the same result of Newman \cite{Newmann1962}, we get another  infinite families of congruences modulo $2$ for $B_{8,3}(n),$
 and $B_{4,5}(n).$	
 \begin{theorem}\label{thmnewman11}
	Let $a_2(n)$ be defined by
	\begin{equation}\label{eq1101}
		\sum_{n=0}^{\infty} a_2(n)q^n:= \left(q;q\right)^6_{\infty} \left(q^3;q^3\right)_{\infty} 
	\end{equation}
	and let $p \geq 5$ be a prime number.
	Define 
	\begin{equation}\label{eq1102}
		w_2(p):= a_2\left( \frac{3(p^2-1)}{8}\right) + p^2 \cdot \left( \frac{-6}{p}\right)_{L} \left( \frac{-3(p^2-1)}{8 p}\right)_{L}
	\end{equation}
	\begin{itemize}
		\item[(i)] If $ w_2(p) \equiv 0 \pmod2,$ then for $n,k \geq 0,$ if $p \nmid n,$ we have
		\begin{equation}\label{eq1103}
			B_{8,3} \left( p^{4k+3}n+  \frac{3 \left( p^{4k+4}-1\right)}{8} \right) \equiv 0 \pmod2.
		\end{equation}
		\item[(ii)] If $ w_2(p) \not \equiv 0 \pmod2,$ then for $n,k \geq 0,$ if $p \nmid n,$ we have
		\begin{equation}\label{eq1104}
			B_{8,3} \left( p^{6k+5}n+  \frac{3 \left( p^{6k+6}-1\right)}{8} \right) \equiv 0 \pmod2.
		\end{equation}
		
		\item[(iii)] If $ w_2(p) \not \equiv 0 \pmod2,$ then for $n,k \geq 0,$ with $ w_2(p)\equiv  p^2 \cdot \left( \frac{ -6n-2+ \frac{(p^2-1)}{4}}{p}\right)_{L}\pmod2,$ we have
		\begin{equation}\label{eq1105}
			B_{8,3} \left( p^{6k+2}n+  \frac{3 \left( p^{6k+2}-1\right)}{8} \right) \equiv 0 \pmod2.
		\end{equation}
	\end{itemize}
\end{theorem}

	\begin{theorem}\label{thmnewman12}
	Let $a_3(n)$ be defined by
	\begin{equation}\label{eq1201}
		\sum_{n=0}^{\infty} a_3(n)q^n:= \left(q;q\right)^2_{\infty} \left(q^5;q^5\right)_{\infty} 
	\end{equation}
	and let $p \geq 5$ be a prime number.
	Define 
	\begin{equation}\label{eq1202}
		w_3(p):= a_3\left( \frac{7(p^2-1)}{24}\right) + \left( \frac{-10}{p}\right)_{L} \left( \frac{-7(p^2-1)}{24 p}\right)_{L}
	\end{equation}
	\begin{itemize}
		\item[(i)] If $ w_3(p) \equiv 0 \pmod2,$ then for $n,k \geq 0,$ if $p \nmid n,$ we have
		\begin{equation}\label{eq1203}
			B_{4,5} \left( p^{4k+3}n+  \frac{7 \left( p^{4k+4}-1\right)}{24} \right) \equiv 0 \pmod2.
		\end{equation}
		\item[(ii)] If $ w_3(p) \not \equiv 0 \pmod2,$ then for $n,k \geq 0,$ if $p \nmid n,$ we have
		\begin{equation}\label{eq1204}
			B_{4,5} \left( p^{6k+5}n+  \frac{7 \left( p^{6k+6}-1\right)}{24} \right) \equiv 0 \pmod2.
		\end{equation}
		
		\item[(iii)] If $ w_3(p) \not \equiv 0 \pmod2,$ then for $n,k \geq 0,$ with $ w_3(p)\equiv  \left( \frac{ -10n-2+ \frac{11(p^2-1)}{12}}{p}\right)_{L}\pmod2,$ we have
		\begin{equation}\label{eq1205}
			B_{4,5} \left( p^{6k+2}n+  \frac{7 \left( p^{6k+2}-1\right)}{24} \right) \equiv 0 \pmod2.
		\end{equation}
	\end{itemize}
\end{theorem}

 	\section{Preliminaries}
	We recall some basic facts and definition on modular forms. For more details, one can see \cite{Koblitz}, \cite{Ono2004}. We start with some matrix groups. We define
	\begin{align*}
		\Gamma:=\mathrm{SL_2}(\mathbb{Z})= &\left\{ \begin{bmatrix}
			a && b \\c && d
		\end{bmatrix}: a, b, c, d \in \mathbb{Z}, ad-bc=1 \right\},\\
		\Gamma_{\infty}:= &\left\{\begin{bmatrix}
			1 &n\\ 0&1	\end{bmatrix}: n \in \mathbb{Z}\right\}.
	\end{align*}
	For a positive integer $N$, we define
	\begin{align*}
		\Gamma_{0}(N):=& \left\{ \begin{bmatrix}
			a && b \\c && d
		\end{bmatrix} \in \mathrm{SL_2}(\mathbb{Z}) : c\equiv0 \pmod N \right\},\\
		\Gamma_{1}(N):=& \left\{ \begin{bmatrix}
			a && b \\c && d
		\end{bmatrix} \in \Gamma_{0}(N) : a\equiv d  \equiv 1 \pmod N \right\}
	\end{align*}
	and 
	\begin{align*}
		\Gamma(N):= \left\{ \begin{bmatrix}
			a && b \\c && d
		\end{bmatrix} \in \mathrm{SL_2}(\mathbb{Z}) : a\equiv d  \equiv 1 \pmod N,  b \equiv c  \equiv 0 \pmod N \right\}.
	\end{align*}
	A subgroup of $\Gamma=\mathrm{SL_2}(\mathbb{Z})$ is called a congruence subgroup if it contains $ \Gamma(N)$ for some $N$ and the smallest $N$ with this property is called its level. Note that $ \Gamma_{0}(N)$ and $ \Gamma_{1}(N)$ are congruence subgroup of level $N,$ whereas $ \mathrm{SL_2}(\mathbb{Z}) $ and $\Gamma_{\infty}$ are congruence subgroups of level $1.$ The index of $\Gamma_0(N)$ in $\Gamma$ is 
	\begin{align*}
		[\Gamma:\Gamma_0(N)]=N\prod\limits_{p|N}\left(1+\frac 1p\right)
	\end{align*}
	where $p$ runs over the prime divisors of $N$.
	
	Let $\mathbb{H}$ denote the upper half of the complex plane $\mathbb{C}$. The group 
	\begin{align*}
		\mathrm{GL_2^{+}}(\mathbb{R}):= \left\{ \begin{bmatrix}
			a && b \\c && d
		\end{bmatrix}: a, b, c, d \in \mathbb{R}, ad-bc>0 \right\},
	\end{align*}
	acts on $\mathbb{H}$ by $ \begin{bmatrix}
		a && b \\c && d
	\end{bmatrix} z = \frac{az+b}{cz+d}.$ We identify $\infty$ with $\frac{1}{0}$ and define $ \begin{bmatrix}
		a && b \\c && d
	\end{bmatrix} \frac{r}{s} = \frac{ar+bs}{cr+ds},$ where $\frac{r}{s} \in \mathbb{Q} \cup \{ \infty\}$. This gives an action of $\mathrm{GL_2^{+}}(\mathbb{R})$ on the extended half plane $\mathbb{H}^{*}=\mathbb{H} \cup \mathbb{Q} \cup \{\infty\}$. Suppose that $\Gamma$ is a congruence subgroup of $\mathrm{SL_2}(\mathbb{Z})$. A cusp of $\Gamma$ is an equivalence class in $\mathbb{P}^{1}=\mathbb{Q} \cup \{\infty\}$ under the action of $\Gamma$.
	
	The group $\mathrm{GL_2^{+}}(\mathbb{R})$ also acts on functions $h:\mathbb{H} \rightarrow \mathbb{C}$. In particular, suppose that $\gamma=\begin{bmatrix}
		a && b \\c && d
	\end{bmatrix}\in \mathrm{GL_2^{+}}(\mathbb{R})$. If $h(z)$ is a meromorphic function on $\mathbb{H}$ and $k$ is an integer, then define the slash operator $|_{k}$ by
	\begin{align*}
		(h|_{k} \gamma)(z):= (\det \gamma)^{k/2} (cz+d)^{-k} h(\gamma z).
	\end{align*}
	
	\begin{definition}
		Let $\Gamma$ be a congruence subgroup of level $N$. A holomorphic function $h:\mathbb{H} \rightarrow \mathbb{C}$ is called a modular form with integer weight $k$ on $\Gamma$ if the following hold:
		\begin{enumerate}[$(1)$]
			\item We have
			\begin{align*}
				h \left( \frac{az+b}{cz+d}\right)=(cz+d)^{k} h(z)
			\end{align*}
			for all $z \in \mathbb{H}$ and $\begin{bmatrix}
				a && b \\c && d
			\end{bmatrix}\in \Gamma$. 
			\item If $\gamma\in SL_2 (\mathbb{Z})$, then $(h|_{k} \gamma)(z)$ has a Fourier expnasion of the form
			\begin{align*}
				(h|_{k} \gamma)(z):= \sum \limits_{n\geq 0}a_{\gamma}(n) q_N^{n}
			\end{align*}
			where $q_N:=e^{2\pi i z /N}$.
		\end{enumerate}
	\end{definition}
	For a positive integer $k$, the complex vector space of modular forms of weight $k$ with respect to a congruence subgroup $\Gamma$ is denoted by $M_{k}(\Gamma)$.
	
	\begin{definition} \cite[Definition 1.15]{Ono2004}
		If $\chi$ is a Dirichlet character modulo $N$, then we say that a modular form $h \in M_{k}(\Gamma_1(N))$ has Nebentypus character $\chi$ if 
		\begin{align*}
			h \left( \frac{az+b}{cz+d}\right)=\chi(d) (cz+d)^{k} h(z)
		\end{align*}
		for all $z \in \mathbb{H}$ and $\begin{bmatrix}
			a && b \\c && d
		\end{bmatrix}\in \Gamma_{0}(N)$. The space of such modular forms is denoted by $M_{k}(\Gamma_0(N), \chi)$.
	\end{definition}
	
	The relevant modular forms for the results obtained in this article arise from eta-quotients. Recall that the Dedekind eta-function $\eta (z)$ is defined by 
	\begin{align*}
		\eta (z):= q^{1/24}(q;q)_{\infty}=q^{1/24} \prod\limits_{n=1}^{\infty} (1-q^n)
	\end{align*}
	where $q:=e^{2\pi i z}$ and $z \in \mathbb{H}$. A function $h(z)$ is called an eta-quotient if it is of the form
	\begin{align*}
		h(z):= \prod\limits_{\delta|N} \eta(\delta z)^{r_{\delta}}
	\end{align*}
	where $N$ and $r_{\delta}$ are integers with $N>0$. 
	
	\begin{theorem} \cite[Theorem 1.64]{Ono2004} \label{thm2.3}
		If $h(z)=\prod\limits_{\delta|N} \eta(\delta z)^{r_{\delta}}$ is an eta-quotient such that $k= \frac 12$ $\sum_{\delta|N} r_{\delta}\in \mathbb{Z}$, 
		\begin{align*}
			\sum\limits_{\delta|N} \delta r_{\delta} \equiv 0\pmod {24}	\quad \textrm{and} \quad \sum\limits_{\delta|N} \frac{N}{\delta}r_{\delta} \equiv 0\pmod {24},
		\end{align*}
		then $h(z)$ satisfies
		\begin{align*}
			h \left( \frac{az+b}{cz+d}\right)=\chi(d) (cz+d)^{k} h(z)
		\end{align*}
		for each $\begin{bmatrix}
			a && b \\c && d
		\end{bmatrix}\in \Gamma_{0}(N)$. Here the character $\chi$ is defined by $\chi(d):= \left(\frac{(-1)^{k}s}{d}\right)$ where $s=\prod_{\delta|N} \delta ^{r_{\delta}}$.
	\end{theorem}
	
	\begin{theorem} \cite[Theorem 1.65]{Ono2004} \label{thm2.4}
		Let $c,d$ and $N$ be positive integers with $d|N$ and $\gcd(c,d)=1$. If $f$ is an eta-quotient satisfying the conditions of Theorem \ref{thm2.3} for $N$, then the order of vanishing of $f(z)$ at the cusp $\frac{c}{d}$ is
		\begin{align*}
			\frac{N}{24}\sum\limits_{\delta|N} \frac{\gcd(d, \delta)^2 r_{\delta}}{\gcd(d, \frac{N}{ d} )d \delta}.
		\end{align*}
	\end{theorem}
	Suppose that $h(z)$ is an eta-quotient satisfying the conditions of Theorem \ref{thm2.3} and that the associated weight $k$ is a positive integer. If $h(z)$ is holomorphic at all of the cusps of $\Gamma_0(N)$, then $h(z) \in M_{k}(\Gamma_0(N), \chi)$. Theorem \ref{thm2.4} gives the necessary criterion for determining orders of an eta-quotient at cusps. In the proofs of our results, we use Theorems \ref{thm2.3} and \ref{thm2.4} to prove that $h(z) \in M_{k}(\Gamma_0(N), \chi)$ for certain eta-quotients $h(z)$ we consider in the sequel.
	
	We shall now mention a result of Serre \cite[P. 43]{Serre1974} which will be used later. 
	
	\begin{theorem}\label{thm2.5}
		Let $h(z) \in M_{k}(\Gamma_{0}(N), \chi)$ has Fourier expansion
		$$ h(z)= \sum_{n=0}^{\infty} b(n) q^n \in \mathbb{Z}[[q]].$$
		Then for a positive integer $r$,  there is a constant $\alpha>0$ such that
		$$\#\{ 0 < n \leq X: b(n) \not \equiv 0 \pmod{r}\} = \mathcal{O}\left( \frac{X}{(\log X)^{\alpha}}\right).$$
		Equivalently 
		\begin{align}\label{2e1}
			\begin{split}
				\lim\limits_{X \to \infty} \frac{\#\{ 0 < n \leq X: b(n) \not \equiv 0 \pmod{r}\}}{X}= 0.
		\end{split}	\end{align}
	\end{theorem}

	We finally recall the definition of Hecke operators and a few relavent results. Let $m$ be a positive integer and $h(z)= \sum \limits_{n= 0}^{\infty}a(n) q^{n}\in M_{k}(\Gamma_0(N), \chi)$. Then the action of Hecke operator $T_m$ on $h(z)$ is defined by
	\begin{align*}
		h(z)|T_{m} := \sum \limits_{n= 0}^{\infty} \left(\sum \limits_{d|\gcd(n,m)} \chi(d) d^{k-1} a\left(\frac{mn}{d^2}\right)\right)q^{n}.
	\end{align*}
	In particular, if $m=p$ is a prime, we have
	\begin{align*}
		h(z)|T_p := \sum \limits_{n= 0}^{\infty}\left( a(pn) + \chi(p) p^{k-1} a\left(\frac{n}{p}\right)\right)q^{n}.
	\end{align*}
	We note that $a(n)=0$ unless $n$ is a non-negative integer.

\section{Proof of Theorem \ref{mainthm16}}
Let $m=p_1^{\al_1} p_2^{\al_2}\cdots p_r^{\al_r}, $ where $p_i \geq 5$ be distinct prime numbers and $\al_i \geq 0$ be non-negative integers.
From \eqref{eq501}, we get
\begin{small}
	\begin{equation}\label{eq601}
		\sum_{n=0}^{\infty} B_{p,m}(n) q^n =  \frac{ (q^{p};q^{p})_{\infty} (q^{m};q^{m})_{\infty}}{(q;q)^2_{\infty}}.
	\end{equation}
\end{small}
Note that for any prime $p$ and positive integers $j,$ we have
\begin{small}
	\begin{equation}\label{eq602}
		(q;q)_{\infty}^{p^j}\equiv (q^p;q^p)_{\infty}^{p^{j-1}} \pmod{p^j}.
	\end{equation}
\end{small}
Thus, we have
\begin{small}
	\begin{equation}\label{eq601a}
		\sum_{n=0}^{\infty} B_{p,m}(n) q^n \equiv  \frac{ (q^{p};q^{p})_{\infty} (q^{m};q^{m})_{\infty}}{(q;q)^2_{\infty}} \equiv (q;q)^{p-2}_{\infty} (q^{m};q^{m})_{\infty} \pmod p .
	\end{equation}
\end{small}
We define $$H_{p}(z):= \frac{ \eta^p(24z)}{  \eta( 24 pz)} .$$
Using the binomial theorem, we get 
\begin{small}
	\begin{equation*}
		H^{p^j}_{p}(z)= \frac{  \eta^{p^{j+1}}(24 z)}{  \eta^{p^j}( 24 pz)}  \equiv 1 \pmod {p^{j+1}}.
	\end{equation*}
\end{small}
Define 
$$F_{p, m,j}(z):= \eta^{p-2}(24z) \eta(24mz) H^{p^j}_{p}(z)= \frac{  \eta^{(p^{j+1}+p-2)}( 24z) \eta(24mz)}{ \eta^{p^j}( 24 pz) } .$$
On modulo $p^{j+1}.$ we get
\begin{small}
	\begin{equation}\label{eq603}
		F_{p, m,j}(z) \equiv \eta^{p-2}(24z) \eta(24mz) \equiv  q^{p-2+m} (q^{24};q^{24})^{p-2}_{\infty} (q^{24m};q^{24m})_{\infty}.
	\end{equation}
\end{small}
Combining \eqref{eq601a} and \eqref{eq603} together, we obtain
\begin{small}
	\begin{equation}\label{eq604a}
		F_{p, m,j}(z)\equiv  q^{p-2+m} (q^{24};q^{24})^{p-2}_{\infty} (q^{24m};q^{24m})_{\infty} \equiv \sum_{n=0}^{\infty} B_{p,m}(n)q^{24n+ p-2+m } \pmod{p}.
	\end{equation}
\end{small}
Next, we show that $F_{p, m,j}(z)$ is a modular form  for certain values of $p, m , \ \hbox{and} \ j$. 

\begin{lemma}\label{lem11}
	Let $m=p_1^{\al_1} p_2^{\al_2}\cdots p_r^{\al_r}, $ where $p_i \geq 5$ be distinct primes and $\al_i \geq 0$ be integers. Then, we have $ F_{p, m,j}(z) \in M_{k}(\Gamma_{0}(N), \chi)$, where $k=\frac{(p-1)}{2} \left(p^j+1\right),   N= 2^5 \cdot 3^2 \cdot  p \cdot m $ and $\chi= \left( \frac{2^{3 \cdot (p-1) (p^j+1)} 3^{(p-1) (p^j+1)} m p^{-p^j}}{\bullet}\right).$
\end{lemma}
\begin{proof}
	Applying Theorem \ref{thm2.3}, we first estimate the level of eta quotient $F_{p, m,j}(z)$ . The level of $ F_{p,m,j}(z) $ is $N= 3 \cdot 2^{3}  \cdot p \cdot m \cdot M,$ where $M$ is the smallest positive integer which satisfies 
	\begin{small}
		\begin{align*}
			3 \cdot 2^{3} \cdot p \cdot m \cdot M & \left[ \frac{p^{j+1}+p-2}{2^{3} \cdot 3} + \frac{1}{2^3 \cdot 3 \cdot m} - \frac{p^{j}}{2^{3} \cdot 3\cdot  p} \right]\equiv 0 \pmod{24} \\
			\implies M & \left[ p \cdot m \cdot \left( p^{j+1}+p-2 \right)  + p  -m p^j \right]\equiv 0 \pmod{24} \\
			\implies M & \left[  m \cdot p^j \left(p^2-1 \right)  + p \cdot m \cdot (p-2) + p \right]\equiv 0 \pmod{24} \\
			\implies 2M & \left[  m \cdot p^j \cdot \frac{\left(p^2-1 \right)}{2}   + p \cdot \frac{ \left( m \cdot (p-2) +1 \right) }{2}  \right]\equiv 0 \pmod{24}.
		\end{align*}
	\end{small}
	Therefore $M=12$ and the level of $F_{p,m,j}(z)$  is $N=2^{5} \cdot 3^2 \cdot p \cdot m $.
	The cusps of $\Gamma_{0}(2^{5} \cdot 3^2 \cdot p \cdot m)$ are given by fractions $\frac{c}{d}$ where $d|2^{5} \cdot 3^2 \cdot p \cdot m$ and $\gcd(c,d)=1.$ By using Theorem \ref{thm2.4}, we obtain that $F_{p, m,j}(z)$ is holomorphic at a cusp $\frac{c}{d}$ if and only if 
	
	\begin{align}\label{eq605}
		&(p^{j+1}+p-2) \frac{  \gcd^2(d, 2^{3} \cdot 3)}{2^{3} \cdot 3} +  \frac{\gcd^2(d, 24m)}{24m} -p^j \frac{\gcd^2(d, 2^{3} \cdot 3 p )}{2^{3} \cdot 3 p} \geq 0 \\ \notag
		& \iff L:=  (p^{j+1}+p-2) \cdot p G_1 + \frac{p}{m} \cdot  G_2 - p^j \geq 0,
	\end{align}
	
	where $G_1= \frac{\gcd^2(d, 2^{3} \cdot 3)}{ \gcd^2(d, 2^{3} \cdot 3 \cdot  p )}$ and  $G_2= \frac{\gcd^2(d,  2^{3} \cdot 3 m)}{\gcd^2(d, 2^{3} \cdot 3 \cdot  p ) } .$ 
	
	Let $d$ be a divisor of $2^5 \cdot 3^2 \cdot  p  \cdot m$. We can write  $d= 2^{r_1}3^{r_2} p^{r_3} t : 0 \leq r_1\leq 5, 0 \leq r_2\leq 2, 0 \leq r_3\leq 1, t|m.$

	Next, we find all possible value of divisors of $2^5 \cdot 3^2 \cdot  p  \cdot m$ to compute equation \eqref{eq605}.
	
	When $d= 2^{r_1}3^{r_2} p^{r_3} t : 0 \leq r_1\leq 5, 0 \leq r_2\leq 2, 0 \leq r_3\leq 1, t|m.$ Then $  \frac{1}{p^2} \leq G_1 \leq 1 $ $  \frac{t^2}{p^2} \leq G_2 \leq 1.$
	Then equation \eqref{eq605} will be
	\begin{align}\label{eq607}
		L \geq     (p^{j+1}+p-2) \cdot p \cdot \frac{1}{p^2}  + \frac{p}{m} \cdot \frac{t^2}{p^2} - p^j \geq \left( 1- \frac{2}{p} \right) + \frac{t^2}{pm} \geq 0.
	\end{align}
	Therefore, $F_{p,m,j}(z)$ is holomorphic at every cusp $\frac{c}{d}.$
	Using Theorem \ref{thm2.3}, we compute the weight of $F_{p, m,j}(z)$ is $k= \frac{(p-1)}{2} \left(p^j+1\right) $ which is a positive integer. The associated character for $F_{p,m,j}(z)$ is $$\chi= \left( \frac{2^{3 \cdot (p-1) (p^j+1)} 3^{(p-1) (p^j+1)} m p^{-p^j}}{\bullet}\right).$$ 
	Thus $ F_{p, m,j}(z) \in  M_{k}(\Gamma_{0}(N), \chi)$ where $k$, $N$ and $\chi$ are as above
\end{proof}
\subparagraph*{\bf Proof of Theorem \ref{mainthm16}}
Using Lemma \ref{lem11}, we find $ F_{p, m,j}(z) \in  M_{k}(\Gamma_{0}(N), \chi)$ where $k=\frac{(p-1)}{2} \left(p^j+1\right),   N= 2^5 \cdot 3^2 \cdot  p \cdot m $ and $\chi= \left( \frac{2^{3 \cdot (p-1) (p^j+1)} 3^{(p-1) (p^j+1)} m p^{-p^j}}{\bullet}\right).$ Employing Theorem \ref{thm2.5}, we obtain that the Fourier coefficients of $F_{p, m,j}(z) $ satisfies \eqref{2e1} for $r=p$ which implies that the Fourier coefficient of $F_{p, m,j}(z) $ are almost always divisible by $p$. Hence, from $\eqref{eq604a}$, we conclude that $B_{p,m}(n)$ are almost always divisible by $p$. This completes the proof of Theorem \ref{mainthm16}.
\qed 
%%%%%%%%%%%%%%%%%%%%%%%%%%%%%%%%%%%%%%%%%%%%%%%%%%%%%%%%%%%%%%%%%%%%%%%%%%%%%%%%
\section{ Congruences for $B_{3,7}(n).$}
	\begin{proof}[{\bf Proof of Theorem \ref{thm8}(i)}] 
		From equation \eqref{eq501}, we have
		\begin{equation}\label{eq800}
			\sum_{n=0}^{\infty} B_{3,7}(n)q^{n} = \frac{(q^3;q^3)_{\infty} (q^{7};q^{7})_{\infty}}{(q;q)^2_{\infty}}.
		\end{equation}
	From \eqref{eq800} and \eqref{eq602}, we get
		\begin{equation}\label{eq800b}
			\sum_{n=0}^{\infty} B_{3,7}(n)q^{n} \equiv  \frac{(q^3;q^3)_{\infty} (q^{7};q^{7})_{\infty}}{(q;q)^2_{\infty}} \equiv (q;q)_{\infty} (q^{7};q^{7})_{\infty} \pmod 3 .
		\end{equation}
		Thus, we have
		\begin{equation}\label{eq801}
			\sum_{n=0}^{\infty} B_{3,7}(n)q^{3n+1} \equiv q (q^3;q^3)_{\infty} (q^{21};q^{21})_{\infty}  \equiv \eta(3z) \eta(21z)\pmod3 .
		\end{equation}
		
		By using Theorem \ref{thm2.3}, we obtain $ \eta(3z) \eta(21z) \in  S_1(\Gamma_{0}(63), \left(\frac{-63}{\bullet}\right)). $ Thus $\eta(3z) \eta(21z) $ has a Fourier expansion i.e.
		\begin{equation}\label{eq802}
			\eta(3z) \eta(21z)= q-q^4 - q^{7}- \cdots= \sum_{n=1}^{\infty}  a(n) q^n 
		\end{equation}
		Thus, $a(n)=0$ if $n \not \equiv 1 \pmod 3,$ for all $n \geq 0.$ From \eqref{eq801} and \eqref{eq802}, comparing the coefficient of $q^{3n+1}$, we get
		\begin{equation}\label{eq803}
			B_{3,7}(n) \equiv a(3n+1) \pmod 3.
		\end{equation}
		Since $ \eta(3z) \eta(21z) $ is a Hecke eigenform (see \cite{Martin1996}), it gives
		$$\eta(3z) \eta(21z)|T_p= \sum_{n=1}^{\infty} \left(a(pn)+  \left(\frac{-63}{p}\right)_{L} a\left(\frac{n}{p}\right)\right) q^n = \lambda(p)\sum_{n=1}^{\infty} a(n)q^n.$$  Note that $ \left(\frac{-63}{p}\right)_{L} = \left(\frac{-7}{p}\right)_{L}. $ Comparing the coefficients of $q^n$ on both sides of the above equation, we get
		\begin{equation}\label{eq804}
			a(pn)+  \left(\frac{-7}{p}\right)_{L} a\left(\frac{n}{p}\right) = \lambda(p) a(n).
		\end{equation}
		Since $a(1)=1$ and $a(\frac{1}{p})=0,$ if we put $n=1$ in the above expression, we get $a(p)=\lambda(p).$ As $a(p)=0$ for all $p \not \equiv 1 \pmod 3$ this implies that $\lambda(p)=0$ for all $p \not \equiv 1 \pmod 3.$ From \eqref{eq804} we get that for all $p \not \equiv 1 \pmod 3 $  
		\begin{equation}\label{eq805}
			a(pn)+  \left(\frac{-7}{p}\right)_{L} a\left(\frac{n}{p}\right) =0.
		\end{equation}
		Now, we consider two cases here. If $p \not| n,$ then 
		replacing $n$ by $pn+r$ with $\gcd(r,p)=1$ in \eqref{eq805}, we get
		\begin{equation}\label{eq806}
			a(p^2n+ pr)=0 .
		\end{equation}
		Replacing $n$ by $3n-pr+1$ in \eqref{eq806} and using \eqref{eq803}, we get 
		\begin{equation}\label{eq807}
			B_{3,7}\left(p^2n +pr \frac{(1- p^2)}{3} +  \frac{p^2-1}{3}\right) \equiv 0 \pmod3.
		\end{equation}
		Since $p \equiv 2 \pmod 3,$ we have $3| (1-p^2)$ and  $\gcd\left( \frac{(1- p^2)}{3} , p\right)=1,$ when $r$ runs over a residue system excluding the multiples of $p$, so does $ r \frac{(1- p^2)}{3}.$
		Thus for $p \nmid j,$ \eqref{eq807} can be written as
		\begin{equation}\label{eq808}
			B_{3,7}\left(p^2n+pj+ \frac{p^2-1}{3}\right) \equiv 0 \pmod3.
		\end{equation}
		Now we consider the second case, when $p |n.$ Here replacing $n$ by $pn$ in \eqref{eq805}, we get
		\begin{equation}\label{eq809}
			a(p^2n) = - \left(\frac{-7}{p}\right)_{L} \cdot a\left(n\right).
		\end{equation}
		Further substituting $n$ by $3n+1$ in \eqref{eq809} we obtain
		\begin{equation}\label{eq810}
			a(3p^2n+ p^2) = - \left(\frac{-7}{p}\right)_{L} \cdot  a\left(3n+1\right).
		\end{equation}
		Using \eqref{eq803} in \eqref{eq810}, we get
		\begin{equation}\label{eq811}
			B_{3,7}\left(p^2n+ \frac{p^2-1}{3} \right) = - \left(\frac{-7}{p}\right)_{L} \cdot B_{3,7} \left(n\right).
		\end{equation}
		Let $p_i$ be primes such that $p_i \equiv 2 \pmod 3.$ Further note that
		$$ p_1^2 p_2^2 \cdots p_k^2 n + \frac{p_1^2 p_2^2 \cdots p_k^2-1}{3}= p_1^2 \left(p_2^2 \cdots p_k^2 n + \frac{ p_2^2 \cdots p_k^2 -1}{3} \right)+ \frac{p_1^2-1}{3} .$$
		Repeatedly using \eqref{eq811} and \eqref{eq808}, we get
		\begin{align*}
			&	B_{3,7} \left( p_1^2 p_2^2 \cdots p_k^2 p_{k+1}^2 n + \frac{p_1^2 p_2^2 \cdots p_k^2 p_{k+1}\left(3j+ p_{k+1}\right)-1}{3} \right) \\
			& \equiv - \left(\frac{-7}{p_1}\right)_{L} \cdot  B_{3,7} \left(p_2^2 \cdots p_k^2 p_{k+1}^2 n + \frac{ p_2^2 \cdots p_k^2 p_{k+1}\left(3j+ p_{k+1}\right) -1}{3} \right) \equiv \cdots\\
			& \equiv (- 1)^k \left(\left(\frac{-7}{p_1}\right)_{L} \left(\frac{-7}{p_2}\right)_{L} \cdots \left(\frac{-7}{p_k}\right)_{L} \right) B_{3,7} \left(p_{k+1}^2 n + p_{k+1} j + \frac{p^2_{k+1}-1}{3} \right) \equiv 0 \pmod3
		\end{align*}
		when $j \not \equiv 0 \pmod{ p_{k+1}}.$ This completes the proof of the theorem.
	\end{proof}
	
	%%%%%%%%%%%%%%%%%%%%%%%%%%%%%%%%%%%%%%%%%%%%%%%%%%%%%%%%%%%%%%%%%%%%%%%%%%%%%%%
%	\section{Congruences for $B_{3,7}(n).$}
	%%%%%%%%%%%%%%%%%%%%%%%%%%%%%%%%%%%%%%%%%%%%%%%%%%%%%%%%%%%%%%%%%%
	
	\begin{proof}[{\bf Proof of Theorem \ref{thmNewmannapplication1}(i)}]
		From \eqref{eq800} and \eqref{eq602}, we get
		\begin{equation}\label{eq800b*}
			\sum_{n=0}^{\infty} B_{3,7}(n)q^{n} \equiv  \frac{(q^3;q^3)_{\infty} (q^{7};q^{7})_{\infty}}{(q;q)^2_{\infty}} \equiv (q;q)_{\infty} (q^{7};q^{7})_{\infty} \pmod 3 .
		\end{equation}
		Define 
		\begin{equation}\label{eq50}
			\sum_{n=0}^{\infty} c(n) q^{n}:= f_1f_7= (q;q)_{\infty} (q^{7};q^{7})_{\infty}.
		\end{equation}
		In 1959, Newman \cite{Newmann1959} proved that if $p$ is a prime with $p \equiv 1 \pmod 6,$ then
		\begin{equation}\label{eq51}
			c\left(pn+ \frac{p-1}{3}\right) = c\left(\frac{p-1}{3}\right)c(n)- (-1)^{\frac{p-1}{2}} \left( \frac{7}{p} \right)_{L} c\left(\frac{n}{p}- \frac{p-1}{3p}\right).
		\end{equation}
		Observe that, if $p \nmid (3n+1), $ then $\left(\frac{n}{p}- \frac{p-1}{3p}\right)$ is not an integer and
		\begin{equation}\label{eq52}
			c\left(\frac{n}{p}- \frac{p-1}{3p}\right) =0 .
		\end{equation}
		From \eqref{eq51} and \eqref{eq52}, we get that if $p \nmid (3n+1), $ then
		\begin{equation}\label{eq53}
			c\left(pn+ \frac{p-1}{3}\right) = c\left(\frac{p-1}{3}\right)c(n) .
		\end{equation}
		Thus, if $p \nmid (3n+1) $ and $c\left(\frac{p-1}{3}\right) \equiv 0 \pmod 3,$ then for $n \geq 0,$
		\begin{equation}\label{eq54}
			c\left(pn+ \frac{p-1}{3}\right) \equiv 0 \pmod 3 .
		\end{equation}
		Replacing $n$ by $ pn+ \frac{p-1}{3}$ in \eqref{eq51}, we get
			\begin{equation}\label{eq55}
			c\left(p^2n+ \frac{p^2-1}{3}\right) = c\left(\frac{p-1}{3}\right)c\left(pn+ \frac{p-1}{3}\right)- (-1)^{\frac{p-1}{2}} \left( \frac{7}{p} \right)_{L} c\left(n\right) .
		\end{equation}
		Observe that from \eqref{eq55} if $c\left(\frac{p-1}{3}\right) \equiv 0 \pmod 3, $ then for $n \geq 0,$ we get
		\begin{equation}\label{eq56}
			c\left(p^2n+ \frac{p^2-1}{3}\right) \equiv - (-1)^{\frac{p-1}{2}} \left( \frac{7}{p} \right)_{L} c\left(n\right) \pmod 3 .
		\end{equation}
		From \eqref{eq56} and by using mathematical induction, we get that if $c\left(\frac{p-1}{3}\right) \equiv 0 \pmod 3, $ then for $n \geq 0, $ and $\alpha \geq 0, $ we get 
		\begin{equation}\label{eq57}
			c\left(p^{2 \alpha}n+ \frac{p^{2 \alpha}-1}{3}\right) \equiv \left( - (-1)^{\frac{p-1}{2}} \left( \frac{7}{p} \right)_{L} \right)^{\alpha} c\left(n\right) \pmod 3 .
		\end{equation}
		Replacing $n$ by $pn+ \frac{p-1}{3}$ in \eqref{eq57} and using \eqref{eq54}, we obtain that if $p \nmid (3n+1)$ and $c\left(\frac{p-1}{3}\right) \equiv 0 \pmod 3,$ then for $n \geq 0$ and $\alpha \geq 0,$ we get
		\begin{equation}\label{eq58}
			c\left(p^{2 \alpha +1}n+ \frac{p^{2 \alpha+1}-1}{3}\right) \equiv  \pmod 3 .
		\end{equation}
		From \eqref{eq50} and \eqref{eq800b*}, we obtain that for $n \geq 0,$
		\begin{equation}\label{eq59}
			B_{3,7}(n) \equiv c(n) \pmod 3.
		\end{equation}
		Thus, combining \eqref{eq59} and \eqref{eq58}, we obtain the result of Theorem. \end{proof}
	
	%%%%%%%%%%%%%%%%%%%%%%%%%%%%%%%%%%%%%%%%%%%%%%%%%%%%%%%%%%%%%%%%%%

%%%%%%%%%%%%%%%%%%%%%%%%%%%%%%%%%%%%%%%%%%%%%%%%%%%%%%%%%%%%%%%%%%%%%%%%%%%%%%%%%%
%	\section{Proof of Theorem \ref{thm9}}
%%%%%%%%%%%%%%%%%%%%%%%%%%%%%%%%%%%%%%%%%%%%%%%%%%%%%%%%%%%%%%%%%%%%%%%%%%%%
	\begin{proof}[{\bf Proof of Theorem \ref{thm9}(i)}] From \eqref{eq805}, we get that for any prime $p \equiv 2 \pmod 3$
		\begin{equation}\label{eq812}
			a(pn)= -\left(\frac{-7}{p}\right)_{L}  a\left(\frac{n}{p}\right).
		\end{equation}  Replacing $n$ by $3n+2,$ we obtain
		\begin{equation}\label{eq813}
			a(3pn+2p)= -\left(\frac{-7}{p}\right)_{L} a\left(\frac{3n+2}{p}\right).
		\end{equation}
		Next substituting $n$ by $p^kn+r$ with $p \nmid r$ in \eqref{eq813}, we obtain
		\begin{equation}\label{eq814}
			a\left(3 \left(p^{k+1}n+pr+ \frac{2p-1}{3} \right)+1\right)= -\left(\frac{-7}{p}\right)_{L}  a\left(3\left(p^{k-1}n+ \frac{3r+2-p}{3p}\right)+1\right).
		\end{equation}
		Note that $\frac{2p-1}{3} $ and $\frac{3r+2-p}{3p}$ are integers. Using \eqref{eq814} and \eqref{eq803}, we get
		\begin{equation}\label{eq815}
			B_{3,7}\left(p^{k+1}n+pr+ \frac{2p-1}{3}\right)\equiv -\left(\frac{-7}{p}\right)_{L}  B_{3,7}\left(p^{k-1}n+ \frac{3r+2-p}{3p} \right) \pmod3.
		\end{equation}
	\end{proof}
	
	\begin{proof}[{\bf Proof of Corollary \ref{coro4}(i)}]Let $p$ be a prime such that $p \equiv 2 \pmod 3.$ Choose a non negative integer $r$ such that $3r+2=p^{2k-1}.$ Substituting $ k$ by $2k-1$ in \eqref{eq815}, we obtain
		\begin{align*}
			B_{3,7}\left(p^{2k}n+ \frac{p^{2k}-1}{3} \right)&\equiv  -\left(\frac{-7}{p}\right)_{L}  B_{3,7}\left(p^{2k-2}n+ \frac{p^{2k-2}-1}{3} \right)\\
			& \equiv \cdots \equiv \left(-1\right)^k\left(\frac{-7}{p}\right)_{L}^k  B_{3,7}(n) \pmod3.
		\end{align*}
	\end{proof}
	
	%%%%%%%%%%%%%%%%%%%%%%%%%%%%%%%%%%%%%%%%%%%%%%%%%%%%%%%%%%%%%%%%%%%%%%
	%%%%%%%%%%%%%%%%%%%%%%%%%%%%%%%%%%%%%%%%%%%%%%%%%%%%%%%%%%%%%%%%%%%%%%%%%%%%%%%%%%%%%%%%%%%%%%%%%%%%%%%%%%%%%%%%%%%%%%%%%%%%%%%%%%%%%%%%%%%%%%%%%%%%%%%%%%%%%%%%%%%	new theorem for $B_5(n)$                      %%%%%%%%%%%%%%%%%%%%%%%%%%%%%
	\section{Congruences for $B_{3,5}(n)$}
	%%%%%%%%%%%%%%%%%%%%%%%%%%%%%%%%%%%%%%%%%%%%%%%%%%%%%%%%%%%%%%%%%%%%%%%%%%%%% 	%%%%%%%%%%%%%%%%%%%%%%%%%%%%%%%%%%%%%%%%%%%%%%%%%%%%%%%%%%%%%%%%%%

	\begin{proof}[{\bf Proof of Theorem \ref{thm8}(ii)}] 
		From equation \eqref{eq501}, we have
		\begin{equation}\label{eq900}
			\sum_{n=0}^{\infty} B_{3,5}(n)q^{n} = \frac{(q^3;q^3)_{\infty} (q^5;q^5)_{\infty}}{(q;q)^2_{\infty}}.
		\end{equation}
		Note that for any prime number $p$, we have
		\begin{equation}\label{eq900a}
			(q^p;q^p)_{\infty} \equiv (q;q)^{p}_{\infty} \pmod p.
		\end{equation}
		From \eqref{eq900} and \eqref{eq900a}, we get
		\begin{equation}\label{eq900b}
			\sum_{n=0}^{\infty} B_{3,5}(n)q^{n} \equiv \frac{(q^3;q^3)_{\infty} (q^5;q^5)_{\infty}}{(q;q)^2_{\infty}}  \equiv  (q;q)_{\infty} (q^5;q^5)_{\infty} \pmod 3.
		\end{equation}
		Thus we have
		\begin{equation}\label{eq901}
			\sum_{n=0}^{\infty} B_{3,5}(n)q^{4n+1} \equiv q (q^4;q^4)_{\infty} (q^{20};q^{20})_{\infty} \equiv \eta(4z) \eta(20z) \pmod  3 .
		\end{equation}
		By using Theorem \ref{thm2.3}, we obtain $\eta(4z) \eta(20z) \in S_1(\Gamma_{0}(80), \left(\frac{-80}{\bullet}\right)_{L} ). $ Observe that $\eta(4z) \eta(20z) $ has a Fourier expansion i.e.
		\begin{equation}\label{eq902}
			\eta(4z) \eta(20z)= q-q^5- q^{9}+ \cdots= \sum_{n=1}^{\infty}  b(n) q^n 
		\end{equation}
		Thus, $b(n)=0$ if $n \not \equiv 1 \pmod 4,$ for all $n \geq 0.$ From \eqref{eq901} and \eqref{eq902}, comparing the coefficient of $q^{4n+1}$, we get
		\begin{equation}\label{eq903}
			B_{3,5}(n) \equiv b(4n+1) \pmod 3.
		\end{equation}
		Since $ \eta(4z) \eta(20z)$ is a Hecke eigenform (see \cite{Martin1996}), it gives
		$$\eta(4z) \eta(20z)|T_p= \sum_{n=1}^{\infty} \left(b(pn)+  \left(\frac{-80}{p}\right)_{L} b\left(\frac{n}{p}\right)\right) q^n = \lambda(p)\sum_{n=1}^{\infty} b(n)q^n.$$  Note that $ \left(\frac{-80}{p}\right)_{L} = \left(\frac{-5}{p}\right)_{L} $ Comparing the coefficients of $q^n$ on both sides of the above equation, we get
		\begin{equation}\label{eq904}
			b(pn)+ \left(\frac{-5}{p}\right)_{L} b\left(\frac{n}{p}\right) = \lambda(p) b(n).
		\end{equation}
		Since $b(1)=1$ and $b(\frac{1}{p})=0,$ if we put $n=1$ in the above expression, we get $b(p)=\lambda(p).$ As $b(p)=0$ for all $p \not \equiv 1 \pmod 4$ this implies that $\lambda(p)=0$ for all $p \not \equiv 1 \pmod 4.$ From \eqref{eq904} we get that for all $p \not \equiv 1 \pmod 4$  
		\begin{equation}\label{eq905}
			b(pn)+ \left(\frac{-5}{p}\right)_{L} b\left(\frac{n}{p}\right) =0.
		\end{equation}
		Now, we consider two cases here. If $p \not| n,$ then 
		replacing $n$ by $pn+r$ with $\gcd(r,p)=1$ in \eqref{eq905}, we get
		\begin{equation}\label{eq906}
			b(p^2n+ pr)=0 .
		\end{equation}
		Replacing $n$ by $4n-pr+1$ in \eqref{eq906} and using \eqref{eq903}, we get 
		\begin{equation}\label{eq907}
			B_{3,5}\left(p^2n +pr \frac{(1- p^2)}{4} +  \frac{p^2-1}{4}\right) \equiv 0 \pmod3.
		\end{equation}
		Since $p \equiv 3 \pmod 4,$ we have $4| (1-p^2)$ and  $\gcd\left( \frac{(1- p^2)}{4} , p\right)=1,$ when $r$ runs over a residue system excluding the multiples of $p$, so does $ r \frac{(1- p^2)}{4}.$
		Thus for $p \nmid j,$ \eqref{eq907} can be written as
		\begin{equation}\label{eq908}
			B_{3,5}\left(p^2n+pj+ \frac{p^2-1}{4}\right) \equiv 0 \pmod3.
		\end{equation}
		Now we consider the second case, when $p |n.$ Here replacing $n$ by $pn$ in \eqref{eq905}, we get
		\begin{equation}\label{eq909}
			b(p^2n) =  -\left(\frac{-5}{p}\right)_{L} b\left(n\right).
		\end{equation}
		Further substituting $n$ by $4n+1$ in \eqref{eq909} we obtain
		\begin{equation}\label{eq910}
			b(4p^2n+ p^2) = -\left(\frac{-5}{p}\right)_{L}  b\left(4n+1\right).
		\end{equation}
		Using \eqref{eq903} in \eqref{eq910}, we get
		\begin{equation}\label{eq911}
			B_{3,5}\left(p^2n+ \frac{p^2-1}{4} \right) \equiv  - \left(\frac{-5}{p}\right)_{L} \cdot B_{3,5} \left(n\right) \pmod3.
		\end{equation}
		Let $p_i$ be primes such that $p_i \equiv 3 \pmod 4.$ Further note that
		$$ p_1^2 p_2^2 \cdots p_k^2 n + \frac{p_1^2 p_2^2 \cdots p_k^2-1}{4}= p_1^2 \left(p_2^2 \cdots p_k^2 n + \frac{ p_2^2 \cdots p_k^2 -1}{4} \right)+ \frac{p_1^2-1}{4} .$$
		Repeatedly using \eqref{eq911} and \eqref{eq908}, we get
		\begin{align*}
			&	B_{3,5} \left( p_1^2 p_2^2 \cdots p_k^2 p_{k+1}^2 n + \frac{p_1^2 p_2^2 \cdots p_k^2 p_{k+1}\left(4j+ p_{k+1}\right)-1}{4} \right) \\
			& \equiv (-1) \left(\frac{-5}{p_1}\right)_{L}  B_{3,5} \left(p_2^2 \cdots p_k^2 p_{k+1}^2 n + \frac{ p_2^2 \cdots p_k^2 p_{k+1}\left(4j+ p_{k+1}\right) -1}{4} \right) \equiv \cdots\\
			& \equiv (-1)^k \left(\left(\frac{-5}{p_1}\right)_{L} \left(\frac{-5}{p_2}\right)_{L} \cdots \left(\frac{-5}{p_k}\right)_{L} \right)  B_{3,5} \left(p_{k+1}^2 n + p_{k+1} j + \frac{p^2_{k+1}-1}{4} \right) \equiv 0 \pmod3
		\end{align*}
		when $j \not \equiv 0 \pmod{ p_{k+1}}.$ This completes the proof of the theorem.
	\end{proof}
	
	\begin{proof}[{\bf Proof of Theorem \ref{thmNewmannapplication1}(ii)}]
	From \eqref{eq800} and \eqref{eq602}, we get
	\begin{equation}\label{eq1800b*}
		\sum_{n=0}^{\infty} B_{3,5}(n)q^{n} \equiv  \frac{(q^3;q^3)_{\infty} (q^{5};q^{5})_{\infty}}{(q;q)^2_{\infty}} \equiv (q;q)_{\infty} (q^{5};q^{5})_{\infty} \pmod 3 .
	\end{equation}
	Define 
	\begin{equation}\label{eq150}
		\sum_{n=0}^{\infty} t(n) q^{n}:= f_1f_5= (q;q)_{\infty} (q^{5};q^{5})_{\infty}.
	\end{equation}
	In 1959, Newman \cite{Newmann1959} proved that if $p$ is a prime with $p \equiv 1 \pmod 4,$ then
	\begin{equation}\label{eq151}
		t\left(pn+ \frac{p-1}{4}\right) = t\left(\frac{p-1}{4}\right)t(n)-  \left( \frac{5}{p} \right)_{L} t\left(\frac{n}{p}- \frac{p-1}{4p}\right) .
	\end{equation}
	Therefore, if $p \nmid (4n+1), $ then $\left(\frac{n}{p}- \frac{p-1}{4p}\right)$ is not an integer and
	\begin{equation}\label{eq152}
		t\left(\frac{n}{p}- \frac{p-1}{4p}\right) =0 .
	\end{equation}
	From \eqref{eq151} and \eqref{eq152} we get that if $p \nmid (4n+1), $ then
	\begin{equation}\label{eq153}
		t\left(pn+ \frac{p-1}{4}\right) = t\left(\frac{p-1}{4}\right)t(n) .
	\end{equation}
	Thus, if $p \nmid (4n+1) $ and $t\left(\frac{p-1}{4}\right) \equiv 0 \pmod 3,$ then for $n \geq 0,$
	\begin{equation}\label{eq154}
		t\left(pn+ \frac{p-1}{4}\right) \equiv 0 \pmod 3 .
	\end{equation}
	Replacing $n$ by $ pn+ \frac{p-1}{4}$ in \eqref{eq151}, we get
	
	\begin{equation}\label{eq155}
		t\left(p^2n+ \frac{p^2-1}{4}\right) = t\left(\frac{p-1}{4}\right)t\left(pn+ \frac{p-1}{4}\right)- \left( \frac{5}{p} \right)_{L} t\left(n\right) .
	\end{equation}
	Observe that from \eqref{eq155} if $t\left(\frac{p-1}{4}\right) \equiv 0 \pmod 3, $ then for $n \geq 0,$ we get
	\begin{equation}\label{eq156}
		t\left(p^2n+ \frac{p^2-1}{4}\right) \equiv -  \left( \frac{5}{p} \right)_{L} t\left(n\right) \pmod 3 .
	\end{equation}
	From \eqref{eq156} and by using mathematical induction, we get that if $t\left(\frac{p-1}{4}\right) \equiv 0 \pmod 3, $ then for $n \geq 0, $ and $\alpha \geq 0, $ we get 
	\begin{equation}\label{eq157}
		t\left(p^{2 \alpha}n+ \frac{p^{2 \alpha}-1}{4}\right) \equiv \left( - \left( \frac{5}{p} \right)_{L} \right)^{\alpha} t\left(n\right) \pmod 3 .
	\end{equation}
	Replacing $n$ by $pn+ \frac{p-1}{4}$ in \eqref{eq157} and using \eqref{eq154}, we obtain that if $p \nmid (4n+1)$ and $t\left(\frac{p-1}{4}\right) \equiv 0 \pmod 3,$ then for $n \geq 0$ and $\alpha \geq 0,$ we get
	\begin{equation}\label{eq158}
		t\left(p^{2 \alpha +1}n+ \frac{p^{2 \alpha+1}-1}{4}\right) \equiv 0 \pmod 3 .
	\end{equation}
	From \eqref{eq150} and \eqref{eq1800b*}, we obtain that for $n \geq 0,$
	\begin{equation}\label{eq159}
		B_{3,5}(n) \equiv t(n) \pmod 3 .
	\end{equation}
	Thus, combining \eqref{eq159} and \eqref{eq158}, we obtain the result of Theorem. 
	\end{proof}

		%%%%%%%%%%%%%%%%%%%%%%%%%%%%%%%%%%%%%%%%%%%%%%%%%%%%%%%%%%%%%%
%	\section{Proof of Theorem \ref{thm11}}
%%%%%%%%%%%%%%%%%%%%%%%%%%%%%%%%%%%%%%%%%%%%%%%%%%%%%%%%%%%%%%%%%%%%%%%%%
	
	\begin{proof}[{\bf Proof of Theorem \ref{thm9}(ii)}] From \eqref{eq905}, we get that for any prime $p \equiv 3 \pmod 4$
		\begin{equation}\label{eq912}
			b(pn)= -\left(\frac{-5}{p}\right)_{L}  b\left(\frac{n}{p}\right).
		\end{equation}  Replacing $n$ by $4n+3,$ we obtain
		\begin{equation}\label{eq913}
			b(4pn+3p)= - \left(\frac{-5}{p}\right)_{L}  b\left(\frac{4n+3}{p}\right).
		\end{equation}
		Next substituting $n$ by $p^kn+r$ with $p \nmid r$ in \eqref{eq813}, we obtain
		\begin{equation}\label{eq914}
			b\left(4 \left(p^{k+1}n+pr+ \frac{3p-1}{4} \right)+1\right)= - \left(\frac{-5}{p}\right)_{L}  b\left(4\left(p^{k-1}n+ \frac{4r+3-p}{4p}\right)+1\right).
		\end{equation}
		Note that $\frac{3p-1}{4} $ and $\frac{4r+3-p}{4p}$ are integers. Using \eqref{eq914} and \eqref{eq903}, we get
		\begin{equation}\label{eq915}
			B_{3,5}\left(p^{k+1}n+pr+ \frac{3p-1}{4}\right)\equiv - \left(\frac{-5}{p}\right)_{L}  B_{3,5}\left(p^{k-1}n+ \frac{4r+3-p}{4p} \right) \pmod3.
		\end{equation}
	\end{proof}
	
	\begin{proof}[{ \bf Proof of Corollary \ref{coro4}(ii)}]Let $p$ be a prime such that $p \equiv 3 \pmod 4.$ Choose a non negative integer $r$ such that $4r+3=p^{2k-1}.$ Substituting $ k$ by $2k-1$ in \eqref{eq915}, we obtain
		\begin{align*}
			B_{3,5}\left(p^{2k}n+ \frac{p^{2k}-1}{4} \right)&\equiv  - \left(\frac{-5}{p}\right)_{L}  B_{3,5}\left(p^{2k-2}n+ \frac{p^{2k-2}-1}{4} \right)\\
			& \equiv \cdots \equiv \left(- \left(\frac{-5}{p}\right)_{L} \right)^k  B_{3,5}(n) \pmod3.
		\end{align*}
	\end{proof}
	%%%%%%%%%%%%%%%%%%%%%%%%%%%%%%%%%%%%%%%%%%%%%%%%%%%%%%%%%%%%%%%%%%%%%%%%%%%%%%%%%%%%%%%	new theorem for $B_5(n)$                      
	%%%%%%%%%%%%%%%%%%%%%%%%%%%%%
	\section{Congruences for $B_{3,2}(n)$}
	%%%%%%%%%%%%%%%%% 
	%%%%%%%%%%%%%%%%%%%%%%%%%%%%%%%%%%%%%%%%%%%%%%%%%%%%%%%%%%%%%%%%%%%%%%%%%%%%	
	%%%%%%%%%%%%%%%%%%%%%%%

	\begin{proof}[{\bf Proof of Theorem \ref{thm8}(iii)}] 
		From equation \eqref{eq501}, we have
		\begin{equation}\label{eq400}
			\sum_{n=0}^{\infty} B_{3,2}(n)q^{n} = \frac{(q^3;q^3)_{\infty} (q^2;q^2)_{\infty}}{(q;q)^2_{\infty}}.
		\end{equation}
	From \eqref{eq400} and \eqref{eq602}, we obtain
		\begin{equation}\label{eq400b}
			\sum_{n=0}^{\infty} B_{3,2}(n)q^{n} = \frac{(q^3;q^3)_{\infty} (q^2;q^2)_{\infty}}{(q;q)^2_{\infty}} \equiv  (q;q)_{\infty} (q^2;q^2)_{\infty} \pmod 3 .
		\end{equation}
		Thus, we have
		\begin{equation}\label{eq401}
			\sum_{n=0}^{\infty} B_{3,2}(n)q^{8n+1} \equiv q (q^8;q^8)_{\infty} (q^{16};q^{16})_{\infty}  \equiv \eta(8z) \eta(16z)\pmod  3 .
		\end{equation}
		
		By using Theorem \ref{thm2.3}, we obtain $\eta(8z) \eta(16z) \in S_1(\Gamma_{0}(128), \left(\frac{-128}{\bullet}\right)_{L} ). $ Note that $\eta(8z) \eta(16z) $ has a Fourier expansion i.e.
		\begin{equation}\label{eq402}
			\eta(8z) \eta(16z)= q-q^9 -2 q^{17}- \cdots= \sum_{n=1}^{\infty}  d(n) q^n 
		\end{equation}
		Thus, $d(n)=0$ if $n \not \equiv 1 \pmod 8,$ for all $n \geq 0.$ From \eqref{eq401} and \eqref{eq402}, comparing the coefficient of $q^{8n+1}$, we get
		\begin{equation}\label{eq403}
			B_{3,2}(n) \equiv d(8n+1) \pmod 3.
		\end{equation}
		Since $ \eta(8z) \eta(16z)$ is a Hecke eigenform (see \cite{Martin1996}), it gives
		$$\eta(8z) \eta(16z)|T_p= \sum_{n=1}^{\infty} \left(d(pn)+  \left(\frac{-128}{p}\right)_{L} d\left(\frac{n}{p}\right)\right) q^n = \lambda(p)\sum_{n=1}^{\infty} d(n)q^n.$$ Note that $ \left(\frac{-128}{p}\right)_{L} = \left(\frac{-2}{p}\right)_{L}.$ Comparing the coefficients of $q^n$ on both sides of the above equation, we get
		\begin{equation}\label{eq404}
			d(pn)+ \left(\frac{-2}{p}\right)_{L}  d\left(\frac{n}{p}\right) = \lambda(p) d(n).
		\end{equation}
		Since $d(1)=1$ and $d(\frac{1}{p})=0,$ if we put $n=1$ in the above expression, we get $d(p)=\lambda(p).$ As $d(p)=0$ for all $p \not \equiv 1 \pmod 8$ this implies that $\lambda(p)=0$ for all $p \not \equiv 1 \pmod 8.$ From \eqref{eq404} we get that for all $p \not \equiv 1 \pmod 8$  
		\begin{equation}\label{eq405}
			d(pn)+ \left(\frac{-2}{p}\right)_{L}  d\left(\frac{n}{p}\right) =0.
		\end{equation}
		Now, we consider two cases here. If $p \not| n,$ then 
		replacing $n$ by $pn+r$ with $\gcd(r,p)=1$ in \eqref{eq405}, we get
		\begin{equation}\label{eq406}
			d(p^2n+ pr)=0 .
		\end{equation}
		Replacing $n$ by $8n-pr+1$ in \eqref{eq406} and using \eqref{eq403}, we get 
		\begin{equation}\label{eq407}
			B_{3,2}\left(p^2n +pr \frac{(1- p^2)}{8} +  \frac{p^2-1}{8}\right) \equiv 0 \pmod3.
		\end{equation}
		Since $p  \not\equiv 1 \pmod 8,$ we have $8| (1-p^2)$ and  $\gcd\left( \frac{(1- p^2)}{8} , p\right)=1,$ when $r$ runs over a residue system excluding the multiples of $p$, so does $ r \frac{(1- p^2)}{8}.$
		Thus for $p \nmid j,$ \eqref{eq407} can be written as
		\begin{equation}\label{eq408}
			B_{3,2}\left(p^2n+pj+ \frac{p^2-1}{8}\right) \equiv 0 \pmod3.
		\end{equation}
		Now we consider the second case, when $p |n.$ Here replacing $n$ by $pn$ in \eqref{eq405}, we get
		\begin{equation}\label{eq409}
			d(p^2n) = -\left(\frac{-2}{p}\right)_{L} \cdot d\left(n\right).
		\end{equation}
		Further substituting $n$ by $8n+1$ in \eqref{eq409} we obtain
		\begin{equation}\label{eq410}
			d(8p^2n+ p^2) = -\left(\frac{-2}{p}\right)_{L}  d\left(8n+1\right).
		\end{equation}
		Using \eqref{eq403} in \eqref{eq410}, we get
		\begin{equation}\label{eq411}
			B_{3,2}\left(p^2n+ \frac{p^2-1}{8} \right) = - \left(\frac{-2}{p}\right)_{L} B_{3,2} \left(n\right).
		\end{equation}
		Let $p_i$ be primes such that $p_i  \not \equiv 1 \pmod 8.$ Further note that
		$$ p_1^2 p_2^2 \cdots p_k^2 n + \frac{p_1^2 p_2^2 \cdots p_k^2-1}{8}= p_1^2 \left(p_2^2 \cdots p_k^2 n + \frac{ p_2^2 \cdots p_k^2 -1}{8} \right)+ \frac{p_1^2-1}{8} .$$
		Repeatedly using \eqref{eq411} and \eqref{eq408}, we get
		\begin{align*}
			&	B_{3,2} \left( p_1^2 p_2^2 \cdots p_k^2 p_{k+1}^2 n + \frac{p_1^2 p_2^2 \cdots p_k^2 p_{k+1}\left(8j+ p_{k+1}\right)-1}{8} \right) \\
			& \equiv -\left(\frac{-2}{p_1}\right)_{L}  B_{3,2} \left(p_2^2 \cdots p_k^2 p_{k+1}^2 n + \frac{ p_2^2 \cdots p_k^2 p_{k+1}\left(8j+ p_{k+1}\right) -1}{8} \right) \equiv \cdots\\
			& \equiv (-1)^k \left(\frac{-2}{p_1}\right)_{L} \left(\frac{-2}{p_2}\right)_{L} \cdots \left(\frac{-2}{p_k}\right)_{L}  B_{3,2} \left(p_{k+1}^2 n + p_{k+1} j + \frac{p^2_{k+1}-1}{8} \right) \equiv 0 \pmod3
		\end{align*}
		when $j \not \equiv 0 \pmod{ p_{k+1}}.$ This completes the proof of the theorem.
	\end{proof}
	
\begin{proof}[{\bf Proof of Theorem \ref{thmNewmannapplication1}(iii)}]
	From \eqref{eq800} and \eqref{eq602}, we get
		\begin{equation}\label{eq2800b*}
			\sum_{n=0}^{\infty} B_{3,2}(n)q^{n} \equiv  \frac{(q^3;q^3)_{\infty} (q^{2};q^{2})_{\infty}}{(q;q)^2_{\infty}} \equiv (q;q)_{\infty} (q^{2};q^{2})_{\infty} \pmod 3 .
		\end{equation}
		Define 
		\begin{equation}\label{eq250}
			\sum_{n=0}^{\infty} u(n) q^{n}:= f_1f_2= (q;q)_{\infty} (q^{2};q^{2})_{\infty}.
		\end{equation}
		
		In 1959, Newman \cite{Newmann1959} proved that if $p$ is a prime with $p \equiv 1 \pmod 8,$ then
		\begin{equation}\label{eq251}
			u\left(pn+ \frac{p-1}{8}\right) = u\left(\frac{p-1}{8}\right)u(n)-  \left( \frac{2}{p} \right)_{L} u\left(\frac{n}{p}- \frac{p-1}{8p}\right) .
		\end{equation}
		Therefore, if $p \nmid (8n+1), $ then $\left(\frac{n}{p}- \frac{p-1}{8p}\right)$ is not an integer and
		\begin{equation}\label{eq252}
			u\left(\frac{n}{p}- \frac{p-1}{8p}\right) =0 .
		\end{equation}
		From \eqref{eq251} and \eqref{eq252} we get that if $p \nmid (8n+1), $ then
		\begin{equation}\label{eq253}
			u\left(pn+ \frac{p-1}{8}\right) = u\left(\frac{p-1}{8}\right)u(n) .
		\end{equation}
		Thus, if $p \nmid (8n+1) $ and $u\left(\frac{p-1}{8}\right) \equiv 0 \pmod 3,$ then for $n \geq 0,$
		\begin{equation}\label{eq254}
			u\left(pn+ \frac{p-1}{8}\right) \equiv 0 \pmod 3 .
		\end{equation}
		Replacing $n$ by $ pn+ \frac{p-1}{8}$ in \eqref{eq251}, we get
		
		\begin{equation}\label{eq255}
			u\left(p^2n+ \frac{p^2-1}{8}\right) = u\left(\frac{p-1}{8}\right)u\left(pn+ \frac{p-1}{8}\right)- \left( \frac{2}{p} \right)_{L} u\left(n\right) .
		\end{equation}
		Observe that from \eqref{eq255} if $u\left(\frac{p-1}{8}\right) \equiv 0 \pmod 3, $ then for $n \geq 0,$ we get
		\begin{equation}\label{eq256}
			u\left(p^2n+ \frac{p^2-1}{8}\right) \equiv -  \left( \frac{2}{p} \right)_{L} u\left(n\right) \pmod 3 .
		\end{equation}
		From \eqref{eq256} and by using mathematical induction, we get that if $u\left(\frac{p-1}{8}\right) \equiv 0 \pmod 3, $ then for $n \geq 0, $ and $\alpha \geq 0, $ we get 
		\begin{equation}\label{eq257}
			u\left(p^{2 \alpha}n+ \frac{p^{2 \alpha}-1}{8}\right) \equiv \left(  \left( \frac{2}{p} \right)_{L} \right)^{\alpha} u\left(n\right) \pmod 3 .
		\end{equation}
		Replacing $n$ by $pn+ \frac{p-1}{8}$ in \eqref{eq157} and using \eqref{eq254}, we obtain that if $p \nmid (8n+1)$ and $u\left(\frac{p-1}{8}\right) \equiv 0 \pmod 3,$ then for $n \geq 0$ and $\alpha \geq 0,$ we get
		\begin{equation}\label{eq258}
			u\left(p^{2 \alpha +1}n+ \frac{p^{2 \alpha+1}-1}{8}\right) \equiv 0 \pmod 3 .
		\end{equation}
		From \eqref{eq250} and \eqref{eq2800b*}, we obtain that for $n \geq 0,$
		\begin{equation}\label{eq259}
			B_{3,2}(n) \equiv u(n) \pmod 3 .
		\end{equation}
		Thus, combining \eqref{eq259} and \eqref{eq258}, we obtain the result of Theorem. 
			\end{proof}
	
%%%%%%%%%%%%%%%%%%%%%%%%%%%%%%%%%%%%%%%%%%%%%%%%%%%%%%%%%%%%%%%%%%%%%%%%%%%%%%%%
%	\section{Proof of Theorem \ref{thm15}}
%%%%%%%%%%%%%%%%%%%%%%%%%%%%%%%%%%%%%%%%%%%%%%%%%%%%%%%%%%%%%%%%%%%%%%%%%%%%%
	
	\begin{proof}[{\bf Proof of Theorem \ref{thm9}(iii)}] From \eqref{eq405}, we get that for any prime $p \equiv t \pmod 8$
		\begin{equation}\label{eq412}
			d(pn)= -\left(\frac{-2}{p}\right)_{L} \cdot d\left(\frac{n}{p}\right).
		\end{equation}  Replacing $n$ by $8n+t,$ we obtain
		\begin{equation}\label{eq413}
			d(8pn+tp)= -\left(\frac{-2}{p}\right)_{L} d\left(\frac{8n+t}{p}\right).
		\end{equation}
		Next substituting $n$ by $p^kn+r$ with $p \nmid r$ in \eqref{eq413}, we obtain
		\begin{equation}\label{eq414}
			d\left(8 \left(p^{k+1}n+pr+ \frac{tp-1}{8} \right)+1\right)= -\left(\frac{-2}{p}\right)_{L}  d \left(8\left(p^{k-1}n+ \frac{8r+t-p}{8p}\right)+1\right).
		\end{equation}
		Note that $\frac{tp-1}{8} $ and $\frac{8r+t-p}{8p}$ are integers. Using \eqref{eq414} and \eqref{eq403}, we get
		\begin{equation}\label{eq415}
			B_{3,2}\left(p^{k+1}n+pr+ \frac{tp-1}{8}\right)\equiv -\left(\frac{-2}{p}\right)_{L}   B_{3,2}\left(p^{k-1}n+ \frac{8r+t-p}{8p} \right) \pmod3.
		\end{equation}
	\end{proof}
	
	\begin{proof}[{\bf Proof of Corollary \ref{coro4}(iii)}]Let $p$ be a prime such that $p \equiv t \pmod 8.$ Choose a non negative integer $r$ such that $8r+t=p^{2k-1}.$ Substituting $ k$ by $2k-1$ in \eqref{eq415}, we obtain
		\begin{align*}
			B_{3,2}\left(p^{2k}n+ \frac{p^{2k}-1}{8} \right)&\equiv  -\left(\frac{2}{p}\right)_{L}  B_{3,2}\left(p^{2k-2}n+ \frac{p^{2k-2}-1}{4} \right)\\
			& \equiv \cdots \equiv \left(-1\right)^k \left(\frac{2}{p}\right)_{L}^k  B_{3,2}(n) \pmod3.
		\end{align*}
	\end{proof}
	\section{Congruences for $B_{7,2}(n)$}
	%%%%%%%%%%%%%%%%% 
	\begin{proof}[{\bf Proof of Theorem \ref{newmannthm4}(i)}] 
		From \eqref{eq501} and \eqref{eq602}, we get
		\begin{equation}\label{eq3800b*}
			\sum_{n=0}^{\infty} B_{7,2}(n)q^{n} \equiv  \frac{(q^7;q^7)_{\infty} (q^{2};q^{2})_{\infty}}{(q;q)^2_{\infty}} \equiv (q;q)^5_{\infty} (q^{2};q^{2})_{\infty} \pmod 7 .
		\end{equation}
		Define 
		\begin{equation}\label{eq350}
			\sum_{n=0}^{\infty} c(n) q^{n}:= f_1^5f_2= (q;q)^5_{\infty} (q^{2};q^{2})_{\infty}.
		\end{equation}
			In 1959, Newman \cite{Newmann1959} proved that if $p$ is a prime with $p \equiv 1 \pmod{24},$ then
		\begin{equation}\label{eq351}
			c\left(pn+ \frac{7(p-1)}{24}\right) = c\left(\frac{7(p-1)}{24}\right)c(n)-  p^{2} \cdot  c\left(\frac{n}{p}- \frac{7(p-1)}{24p}\right) .
		\end{equation}
		Therefore, if $p \nmid (24n+7), $ then $\left(\frac{n}{p}- \frac{7(p-1)}{24p}\right)$ is not an integer and
		\begin{equation}\label{eq352}
			c\left(\frac{n}{p}- \frac{7(p-1)}{24p}\right)=0 .
		\end{equation}
		From \eqref{eq351} and \eqref{eq352} we get that if $p \nmid (24n+7), $ then
		\begin{equation}\label{eq353}
			c\left(pn+ \frac{7(p-1)}{24}\right) = c\left(\frac{7(p-1)}{24}\right)c(n).
		\end{equation}
		Thus, if $p \nmid (24n+7) $ and $c\left(\frac{7(p-1)}{24}\right) \equiv 0 \pmod 7,$ then for $n \geq 0,$
		\begin{equation}\label{eq354}
			c\left(pn+ \frac{7(p-1)}{24}\right) \equiv 0 \pmod 7 .
		\end{equation}
		Replacing $n$ by $ pn+ \frac{7(p-1)}{24}$ in \eqref{eq351}, we get
		
		\begin{equation}\label{eq355}
			c\left(p^2n+ \frac{7(p^2-1)}{24}\right) = c\left(\frac{7(p-1)}{24}\right)c\left(pn+ \frac{7(p-1)}{24}\right)- p^2 \cdot  c\left(n\right) .
		\end{equation}
		Observe that from \eqref{eq355} if $c\left(\frac{7(p-1)}{24}\right) \equiv 0 \pmod 7, $ then for $n \geq 0,$ we get
		\begin{equation}\label{eq356}
			c\left(p^2n+ \frac{7(p^2-1)}{24}\right) \equiv -  p^2\cdot  c\left(n\right) \pmod 7 .
		\end{equation}
		From \eqref{eq356} and by using mathematical induction, we get that if $c\left(\frac{7(p-1)}{24}\right) \equiv 0 \pmod 7,$ then for $n \geq 0, $ and $\alpha \geq 0, $ we get 
		\begin{equation}\label{eq357}
			c\left(p^{2 \alpha}n+ \frac{7(p^{2 \alpha}-1)}{24}\right) \equiv \left( - p^2 \cdot  \right)^{\alpha} c\left(n\right) \pmod 7 .
		\end{equation}
		Replacing $n$ by $pn+ \frac{7(p-1)}{24}$ in \eqref{eq357} and using \eqref{eq354}, we obtain that if $p \nmid (24n+7)$ and $c\left(\frac{7(p-1)}{24}\right) \equiv 0 \pmod 7,$ then for $n \geq 0$ and $\alpha \geq 0,$ we get
		\begin{equation}\label{eq358}
			c\left(p^{2 \alpha +1}n+ \frac{7(p^{2 \alpha+1}-1)}{24}\right) \equiv 0 \pmod 7.
		\end{equation}
		From \eqref{eq350} and \eqref{eq3800b*}, we obtain that for $n \geq 0,$
		\begin{equation}\label{eq359}
			B_{7,2}(n) \equiv c(n) \pmod 7 .
		\end{equation}
		Thus, combining \eqref{eq359} and \eqref{eq358}, we obtain the result of Theorem. 
	\end{proof}
	
	\section{Congruences for $B_{11,2}(n)$}
	%%%%%%%%%%%%%%%%% 
	\begin{proof}[{\bf Proof of Theorem \ref{newmannthm4}(ii)}]
		From \eqref{eq501} and \eqref{eq602}, we get
		\begin{equation}\label{eq4800b*}
			\sum_{n=0}^{\infty} B_{11,2}(n)q^{n} =  \frac{(q^{11};q^{11})_{\infty} (q^{2};q^{2})_{\infty}}{(q;q)^2_{\infty}} \equiv (q;q)^9_{\infty} (q^{2};q^{2})_{\infty} \pmod {11} .
		\end{equation}
		Define 
		\begin{equation}\label{eq450}
			f_1^9f_2= (q;q)^9_{\infty} (q^{2};q^{2})_{\infty} :=\sum_{n=0}^{\infty} f(n) q^{n}.
		\end{equation}
		
		In 1959, Newman \cite{Newmann1959} proved that if $p$ is a prime with $p \equiv 1 \pmod{24},$ then
		\begin{equation}\label{eq451}
			f\left(pn+ \frac{11(p-1)}{24}\right) = f\left(\frac{11(p-1)}{24}\right)f(n)-  p^{4} \cdot  f\left(\frac{n}{p}- \frac{11(p-1)}{24p}\right) .
		\end{equation}
		
		Therefore, if $p \nmid (24n+1), $ then $\left(\frac{n}{p}- \frac{11(p-1)}{24p}\right)$ is not an integer and
		\begin{equation}\label{eq452}
			f\left(\frac{n}{p}- \frac{11(p-1)}{24p}\right)=0 .
		\end{equation}
		From \eqref{eq451} and \eqref{eq452} we get that if $p \nmid (24n+11), $ then
		\begin{equation}\label{eq453}
			f\left(pn+ \frac{11(p-1)}{24}\right) = f\left(\frac{11(p-1)}{24}\right)f(n).
		\end{equation}
		Thus, if $p \nmid (24n+11) $ and $f\left(\frac{11(p-1)}{24}\right) \equiv 0 \pmod{11},$ then for $n \geq 0,$
		\begin{equation}\label{eq454}
			f\left(pn+ \frac{11(p-1)}{24}\right) \equiv 0 \pmod{11} .
		\end{equation}
		Replacing $n$ by $ pn+ \frac{11(p-1)}{24}$ in \eqref{eq451}, we get
		
		\begin{equation}\label{eq455}
			f\left(p^2n+ \frac{11(p^2-1)}{24}\right) = f\left(\frac{11(p-1)}{24}\right)f\left(pn+ \frac{11(p-1)}{24}\right)- p^4 \cdot  f\left(n\right) .
		\end{equation}
		Observe that from \eqref{eq455} if $f\left(\frac{11(p-1)}{24}\right) \equiv 0 \pmod{11}, $ then for $n \geq 0,$ we get
		\begin{equation}\label{eq456}
			f\left(p^2n+ \frac{11(p^2-1)}{24}\right) \equiv -  p^4\cdot  f\left(n\right) \pmod{11} .
		\end{equation}
		From \eqref{eq456} and by using mathematical induction, we get that if $f\left(\frac{11(p-1)}{24}\right) \equiv 0 \pmod{11},$ then for $n \geq 0, $ and $\alpha \geq 0, $ we get 
		\begin{equation}\label{eq457}
			f\left(p^{2 \alpha}n+ \frac{11(p^{2 \alpha}-1)}{24}\right) \equiv \left( - p^4   \right)^{\alpha} \cdot f\left(n\right) \pmod{11} .
		\end{equation}
		Replacing $n$ by $pn+ \frac{11(p-1)}{24}$ in \eqref{eq457} and using \eqref{eq454}, we obtain that if $p \nmid (24n+11)$ and $f\left(\frac{11(p-1)}{24}\right) \equiv 0 \pmod{11},$ then for $n \geq 0$ and $\alpha \geq 0,$ we get
		\begin{equation}\label{eq458}
			f\left(p^{2 \alpha +1}n+ \frac{11(p^{2 \alpha+1}-1)}{24}\right) \equiv  \pmod {11}.
		\end{equation}
		From \eqref{eq450} and \eqref{eq4800b*}, we obtain that for $n \geq 0,$
		\begin{equation}\label{eq459}
			B_{11,2}(n) \equiv f(n) \pmod{11} .
		\end{equation}
		Thus, combining \eqref{eq459} and \eqref{eq458}, we obtain the result of Theorem.
	\end{proof}
	\section{Congruences for $B_{13,2}(n)$}
	%%%%%%%%%%%%%%%%% 
	\begin{proof}[{\bf Proof of Theorem \ref{newmannthm4}(iii)}]
		From \eqref{eq501} and \eqref{eq602}, we get
		\begin{equation}\label{eq5800b*}
			\sum_{n=0}^{\infty} B_{13,2}(n)q^{n} =  \frac{(q^{13};q^{13})_{\infty} (q^{2};q^{2})_{\infty}}{(q;q)^2_{\infty}} \equiv (q;q)^{11}_{\infty} (q^{2};q^{2})_{\infty} \pmod {13} .
		\end{equation}
		Define 
		\begin{equation}\label{eq550}
			f_1^{11}f_2= (q;q)^{11}_{\infty} (q^{2};q^{2})_{\infty} :=\sum_{n=0}^{\infty} g(n) q^{n}.
		\end{equation}
		
		In 1959, Newman \cite{Newmann1959} proved that if $p$ is a prime with $p \equiv 1 \pmod{24},$ then
		\begin{equation}\label{eq551}
			g\left(pn+ \frac{13(p-1)}{24}\right) = g\left(\frac{13(p-1)}{24}\right)g(n)-  p^{5} \cdot  g\left(\frac{n}{p}- \frac{13(p-1)}{24p}\right) .
		\end{equation}
		
		Therefore, if $p \nmid (24n+13), $ then $\left(\frac{n}{p}- \frac{13(p-1)}{24p}\right)$ is not an integer and
		\begin{equation}\label{eq552}
			g\left(\frac{n}{p}- \frac{13(p-1)}{24p}\right)=0 .
		\end{equation}
		
		From \eqref{eq551} and \eqref{eq552} we get that if $p \nmid (24n+13), $ then
		\begin{equation}\label{eq553}
			g\left(pn+ \frac{13(p-1)}{24}\right) = g\left(\frac{13(p-1)}{24}\right)g(n).
		\end{equation}
		Thus, if $p \nmid (24n+13) $ and $g\left(\frac{13(p-1)}{24}\right) \equiv 0 \pmod{13},$ then for $n \geq 0,$
		\begin{equation}\label{eq554}
			g\left(pn+ \frac{13(p-1)}{24}\right) \equiv 0 \pmod{13} .
		\end{equation}
		
		Replacing $n$ by $ pn+ \frac{13(p-1)}{24}$ in \eqref{eq551}, we get
		
		\begin{equation}\label{eq555}
			g\left(p^2n+ \frac{13(p^2-1)}{24}\right) = g\left(\frac{13(p-1)}{24}\right)g\left(pn+ \frac{13(p-1)}{24}\right)- p^5 \cdot  g\left(n\right) .
		\end{equation}
		Observe that from \eqref{eq555} if $g\left(\frac{11(p-1)}{24}\right) \equiv 0 \pmod{13}, $ then for $n \geq 0,$ we get
		\begin{equation}\label{eq556}
			g\left(p^2n+ \frac{13(p^2-1)}{24}\right) \equiv -  p^5\cdot  g\left(n\right) \pmod{13} .
		\end{equation}
		From \eqref{eq556} and by using mathematical induction, we get that if $g\left(\frac{13(p-1)}{24}\right) \equiv 0 \pmod{13},$ then for $n \geq 0, $ and $\alpha \geq 0, $ we get 
		\begin{equation}\label{eq557}
			g\left(p^{2 \alpha}n+ \frac{11(p^{2 \alpha}-1)}{24}\right) \equiv \left( - p^5 \cdot  \right)^{\alpha} g\left(n\right) \pmod{13} .
		\end{equation}
		Replacing $n$ by $pn+ \frac{13(p-1)}{24}$ in \eqref{eq557} and using \eqref{eq554}, we obtain that if $p \nmid (24n+13)$ and $g\left(\frac{13(p-1)}{24}\right) \equiv 0 \pmod{13},$ then for $n \geq 0$ and $\alpha \geq 0,$ we get
		\begin{equation}\label{eq558}
			g\left(p^{2 \alpha +1}n+ \frac{13(p^{2 \alpha+1}-1)}{24}\right) \equiv  \pmod {13}.
		\end{equation}
		From \eqref{eq550} and \eqref{eq4800b*}, we obtain that for $n \geq 0,$
		\begin{equation}\label{eq559}
			B_{13,2}(n) \equiv e(n) \pmod{13} .
		\end{equation}
		Thus, combining \eqref{eq559} and \eqref{eq558}, we obtain the result of Theorem. 
	\end{proof}
	
	\section{Application of Newmann identity II}

	We are going to state a result of Newman [ \cite{Newmann1962}, Theorem $3$ ]
	which play an important role in the proof of our Theorems. We state his result as a lemma. We need the following notations to state the lemma.
	Let $r,$ $s$ be integers such that $r,s \neq 0,$ $r \not \equiv s \pmod2.$ Set
	\begin{equation*}
		\phi(\tau)= \prod_{n=1}^{\infty} \left( 1- x^n\right)^r \left( 1-x^{nq}\right)^s = \sum_{n=0}^{\infty} a(n) x^n
	\end{equation*}
	$ \epsilon = \frac{r+s}{2}, t = \frac{r+sq}{24}, \Delta = t \left(p^2-1\right)= \frac{(r+sq) (p^2-1)}{24} , \theta = (-1)^{\frac{1}{2}- \epsilon} 2 q^s.$ 
	\begin{lemma}\label{lem100}
		With the notations defined as above, the coefficients $a(n)$ of $ \phi(\tau) $ satisfy
		\begin{equation}\label{eq1000}
			a\left(np^2+ \Delta \right)= \gamma_{0}(n)a(n)- p^{2 \epsilon -2} a \left(\frac{n- \Delta}{p^2}\right),
		\end{equation}
		where 
		\begin{equation}\label{eq1000a}
			\gamma_{0}(n)= p^{2 \epsilon -2}c - \left( \frac{\theta }{p}  \right)_{L} p^{ \epsilon -3/2} \left( \frac{n- \Delta }{p}  \right)_{L},
		\end{equation}
		where $c$ is a constant.
	\end{lemma}
	%%%%%%%%%%%%%%%%%%%%%%%%%%%%%%%%%%%%%%%%%%%%%%%%%%%%%%%%%%%%%%%%%%%%%%%%%%%
	%%%%%%%%%%%Newman Identity II
	\section{Congruences for $B_{4,3}(n)$}
	%%%%%%%%%%%%%%%%%%%%%%%%%%%%%%%%%%%%%%%%%%%%%%%%%%%%%%%%%%%%%%%%%%%%%%%%%%%%%%%%
	\begin{proof}[{\bf Proof of Theorem \ref{thmnewman10}}]
		Note that 
		\begin{equation}\label{eq1006}
			\sum_{n=0}^{\infty} B_{4,3}(n)q^n=  \frac{(q^4;q^4)_{\infty} (q^3;q^3)_{\infty}}{(q;q)^2_{\infty}} \equiv (q;q)^2_{\infty} (q^3;q^3)_{\infty} \pmod2.
		\end{equation} 
		Let $a_1(n)$ be defined by \eqref{eq1001}. Substituting $q=3, r=2, s=1$ in Lemma \ref{lem100}, we obtain
		\begin{equation}\label{eq1007}
			a_1\left( p^2n+ \frac{5(p^2-1)}{24}\right)= \gamma_{1}(n) a_1(n)- p \cdot a_1\left( \frac{n- \frac{5(p^2-1)}{24}}{p^2}\right),
		\end{equation}
		where 
		\begin{equation}\label{eq1008}
			\gamma_{1}(n)= p \cdot c - \left( \frac{-6}{p} \right)_{L}  \left( \frac{n- \frac{5(p^2-1)}{24}}{p} \right)_{L}.
		\end{equation}
		If we put $n=0$ in \eqref{eq1007} and using the facts that $a_1(0)=1$ and $a_1\left( \frac{- \frac{5(p^2-1)}{24}}{p^2}\right) =0, $ we get
		\begin{equation}\label{eq1009}
			\gamma_{1}(0)= a_1\left( \frac{5(p^2-1)}{24}\right).
		\end{equation}
		Further taking $n=0$ in \eqref{eq1008} and using \eqref{eq1009}, we obtain
		\begin{equation}\label{eq1010}
			pc= w_1(p),
		\end{equation}
		where $ w_1(p) :=  a_1\left( \frac{5(p^2-1)}{24}\right) + \left( \frac{-6}{p}\right)_{L} \left( \frac{-5(p^2-1)}{24 p}\right)_{L}.$
		Combining \eqref{eq1010}, \eqref{eq1007} and \eqref{eq1008} together, we get
		\begin{equation}\label{eq1011}
			a_1\left( p^2n+ \frac{5(p^2-1)}{24} \right)= \left[ w_1(p) - \left( \frac{-6}{p} \right)_{L}  \left( \frac{n- \frac{5(p^2-1)}{24}}{p} \right)_{L} \right] a_1(n)- p \cdot a_1\left( \frac{n- \frac{5(p^2-1)}{24}}{p^2}\right).
		\end{equation}
		Replacing $n$ by $ pn+ \frac{5(p^2-1)}{24} $ in \eqref{eq1011}, we obtain
		\begin{equation}\label{eq1012}
			a_1\left( p^3n+ \frac{5(p^4-1)}{24} \right)=  w_1(p) a_1\left( pn+ \frac{5(p^2-1)}{24} \right)- p \cdot a_1\left( \frac{n}{p}\right).
		\end{equation}
		If $ w_1(p) \equiv 0 \pmod2,$ then 
		\begin{equation}\label{eq1013}
			a_1\left( p^3n+ \frac{5(p^4-1)}{24} \right) \equiv - p \cdot a_1\left( \frac{n}{p}\right) \pmod2.
		\end{equation}
		Replacing $n$ by $pn$ we have,
		\begin{equation}\label{eq1014}
			a_1\left( p^4n+ \frac{5(p^4-1)}{24} \right) \equiv - p \cdot a_1\left(n \right) \pmod2.
		\end{equation}
		Using mathematical induction in \eqref{eq1014}, we obtain for $n,k \geq 0,$
		\begin{equation}\label{eq1015}
			a_1\left( p^{4k}n+ \frac{5(p^{4k}-1)}{24} \right) \equiv (-p)^k \cdot a_1\left(n \right) \pmod2.
		\end{equation}
		Congruence \eqref{eq1013} implies that If $ w_1(p) \equiv 0 \pmod2$ and $ p \nmid n,$ then
		\begin{equation}\label{eq1016}
			a_1\left( p^3n+ \frac{5(p^4-1)}{24} \right) \equiv 0 \pmod2.
		\end{equation}
		Replacing $n$ by $ p^3n+ \frac{5(p^4-1)}{24}$ in \eqref{eq1015} and using \eqref{eq1016}, we get that if $ w_1(p) \equiv 0 \pmod2$ and $ p \nmid n,$ then
		\begin{equation}\label{eq1017}
			a_1\left( p^{4k+3}n+ \frac{5(p^{4k+4}-1)}{24} \right) \equiv 0 \pmod2.
		\end{equation}
		From \eqref{eq1001}, \eqref{eq1006} and \eqref{eq1017} we obtain that, if $ w_1(p) \equiv 0 \pmod2$ and $ p \nmid n,$ then
		\begin{equation}\label{eq1017a}
			B_{4,3}\left( p^{4k+3}n+ \frac{5(p^{4k+4}-1)}{24} \right) \equiv 0 \pmod2.
		\end{equation}
		Further, replacing $n$ by  $ p^2n+ \frac{5p(p^2-1)}{24}$ in \eqref{eq1012}, we get
		\begin{equation}\label{eq1018}
			a_1\left( p^5n+ \frac{5(p^6-1)}{24} \right)=  w_1(p) a_1\left( p^3n+ \frac{5(p^4-1)}{24} \right)- p \cdot a_1\left( pn+ \frac{5(p^2-1)}{24} \right).
		\end{equation}
		Combining \eqref{eq1018} and \eqref{eq1012} together, we get
		\begin{equation}\label{eq1019}
			a_1\left( p^5n+ \frac{5(p^6-1)}{24} \right)=  \left( w_1^2(p) -p  \right)  a_1\left( pn+ \frac{5(p^2-1)}{24} \right)-  w_1(p) \cdot p \cdot  a_1\left( \frac{n}{p} \right).
		\end{equation}
		Note that if $ w_1(p) \not \equiv 0 \pmod2,$ then  $ w_1(p) \equiv 1 \pmod2$ and hence $ \left( w_1^2(p) -p  \right) \equiv 0 \pmod 2.$ Thus, if $ w_1(p) \not \equiv 0 \pmod2,$ then \eqref{eq1019} will be
		
		\begin{equation}\label{eq1020}
			a_1\left( p^5n+ \frac{5(p^6-1)}{24} \right) \equiv -  w_1(p) \cdot p \cdot  a_1\left( \frac{n}{p} \right) \pmod2.
		\end{equation}
		Therefore, if $p \nmid n,$ then $a_1\left( \frac{n}{p} \right) =0$ and 
		\begin{equation}\label{eq1021}
			a_1\left( p^5n+ \frac{5(p^6-1)}{24} \right) \equiv 0 \pmod2.
		\end{equation}
		Replacing $n$ by $pn$ in \eqref{eq1020}, we obtain
		\begin{equation}\label{eq1022}
			a_1\left( p^6n+ \frac{5(p^6-1)}{24} \right) \equiv -  w_1(p) \cdot p \cdot  a_1\left( n \right) \pmod2.
		\end{equation}
		By using mathematical induction in \eqref{eq1022}, we get
		\begin{equation}\label{eq1023}
			a_1\left( p^{6k}n+ \frac{5(p^{6k}-1)}{24} \right) \equiv  ( - w_1(p) \cdot p)^k \cdot  a_1\left( n \right) \pmod2.
		\end{equation}
		Replacing $n$ by $ p^5n+ \frac{5(p^6-1)}{24}$ in \eqref{eq1023} and using \eqref{eq1021}, it follows that if $ w_1(p) \not \equiv 0 \pmod2$ and $p \nmid n,$ we get 
		\begin{equation}\label{eq1024}
			a_1\left( p^{6k+5}n+ \frac{5(p^{6k+6}-1)}{24} \right) \equiv  0 \pmod2.
		\end{equation}
		From \eqref{eq1024} and \eqref{eq1006} we obtain that, if $ w_1(p) \not \equiv 0 \pmod2$ and $p \nmid n,$  then
		\begin{equation}\label{eq1025}
			B_{4,3}\left( p^{6k+5}n+ \frac{5(p^{6k+6}-1)}{24} \right) \equiv  0 \pmod2.
		\end{equation}
		Note that $\left( \frac{-6}{p} \right)_{L}  \left( \frac{n- \frac{5(p^2-1)}{24}}{p} \right)_{L} = \left( \frac{-6n-1 + \frac{(p^2-1)}{4}}{p} \right)_{L}. $
		From \eqref{eq1011} it follows that if $ w_1(p) \equiv \left( \frac{-6n-1 + \frac{(p^2-1)}{4}}{p} \right)_{L} \pmod2, $ we obtain
		\begin{equation}\label{eq1026}
			a_1\left( p^2n+ \frac{5(p^2-1)}{24} \right) \equiv - p \cdot a_1\left( \frac{n- \frac{5(p^2-1)}{24}}{p^2}\right) \pmod2.
		\end{equation}
		Observe that if $ \left( \frac{-6n-1 + \frac{(p^2-1)}{4}}{p} \right)_{L}  \not \equiv 0 \pmod2,$ then $p \nmid (24n+5)$ and hence $ \left( \frac{n- \frac{5(p^2-1)}{24}}{p}\right) $ and $\left( \frac{n- \frac{5(p^2-1)}{24}}{p^2}\right) $ are not integers. Thus, for $n \geq 0,$ from \eqref{eq1026} 
		we have
		\begin{equation}\label{eq1027}
			a_1\left( p^2n+ \frac{5(p^2-1)}{24} \right) \equiv 0 \pmod2.
		\end{equation}
		Replacing $n$ by $p^2n+ \frac{5(p^2-1)}{24} $ in \eqref{eq1023} and using \eqref{eq1027}, we obtain that if  $w_1(p) \not \equiv 0 \pmod2, $ and $ w_1(p) \equiv \left( \frac{-6n-1 + \frac{(p^2-1)}{4}}{p} \right)_{L} \pmod2,$ then
		
		\begin{equation}\label{eq1028}
			a_1\left( p^{6k+2}n+ \frac{5(p^{6k+2}-1)}{24} \right) \equiv 0 \pmod2.
		\end{equation}
		Combining \eqref{eq1006} and \eqref{eq1028} together, we get
		\begin{equation}\label{eq1029}
			B_{4,3}\left( p^{6k+2}n+ \frac{5(p^{6k+2}-1)}{24} \right) \equiv 0 \pmod2.
		\end{equation}
		\end{proof} 
	
	%%%%%%%%%%%%%%%%%%%%%%%%%%%%%%%%%%%%%%%%%%%%%%%%%%%%%%%%%%%%%%%%%%%%%%%%%%%
	\section{Congruences for $B_{8,3}(n)$}
	%%%%%%%%%%%%%%%%%%%%%%%%%%%%%%%%%%%%%%%%%%%%%%%%%%%%%%%%%%%%%%%%%%%%%%%%%%%%%%%%
	
\begin{proof}[{\bf Proof of Theorem \ref{thmnewman11} }]
		Note that 
		\begin{equation}\label{eq1106}
			\sum_{n=0}^{\infty} B_{8,3}(n)q^n=  \frac{(q^8;q^8)_{\infty} (q^3;q^3)_{\infty}}{(q;q)^2_{\infty}} \equiv (q;q)^6_{\infty} (q^3;q^3)_{\infty} \pmod2.
		\end{equation} 
		Let $a_2(n)$ be defined by \eqref{eq1001}. Substituting $q=3, r=6, s=1$ in Lemma \ref{lem100}, we obtain
		\begin{equation}\label{eq1107}
			a_2\left( p^2n+ \frac{3(p^2-1)}{8}\right)= \gamma_{2}(n) a_2(n)- p^5 \cdot a_2\left( \frac{n- \frac{3(p^2-1)}{8}}{p^2}\right),
		\end{equation}
		where 
		\begin{equation}\label{eq1108}
			\gamma_{2}(n)= p^5 \cdot c - p^2 \cdot \left( \frac{-6}{p} \right)_{L}  \left( \frac{n- \frac{3(p^2-1)}{8}}{p} \right)_{L}.
		\end{equation}
		If we put $n=0$ in \eqref{eq1107} and using the facts that $a_2(0)=1$ and $a_2\left( \frac{- \frac{3(p^2-1)}{8}}{p^2}\right) =0, $ we get
		\begin{equation}\label{eq1109}
			\gamma_{2}(0)= a_2\left( \frac{3(p^2-1)}{8}\right).
		\end{equation}
		Further taking $n=0$ in \eqref{eq1108} and using \eqref{eq1109}, we obtain
		\begin{equation}\label{eq1110}
			p^5c= w_2(p),
		\end{equation}
		where $ w_2(p) :=  a_2\left( \frac{3(p^2-1)}{8}\right) + p^2 \cdot \left( \frac{-6}{p}\right)_{L} \left( \frac{-3(p^2-1)}{8 p}\right)_{L}.$
		Combining \eqref{eq1110}, \eqref{eq1107} and \eqref{eq1108} together, we get
		\begin{equation}\label{eq1111}
			a_2\left( p^2n+ \frac{3(p^2-1)}{8} \right)= \left[ w_2(p) - p^2 \cdot \left( \frac{-6}{p} \right)_{L}  \left( \frac{n- \frac{3(p^2-1)}{8}}{p} \right)_{L} \right] a_2(n)- p^5 \cdot a_2\left( \frac{n- \frac{3(p^2-1)}{8}}{p^2}\right).
		\end{equation}
		Replacing $n$ by $ pn+ \frac{3(p^2-1)}{8} $ in \eqref{eq1111}, we obtain
		\begin{equation}\label{eq1112}
			a_2\left( p^3n+ \frac{3(p^4-1)}{8} \right)=  w_2(p) a_2\left( pn+ \frac{3(p^2-1)}{8} \right)- p^5 \cdot a_2\left( \frac{n}{p}\right).
		\end{equation}
		If $ w_2(p) \equiv 0 \pmod2,$ then 
		\begin{equation}\label{eq1113}
			a_2\left( p^3n+ \frac{3(p^4-1)}{8} \right) \equiv - p^5 \cdot a_2\left( \frac{n}{p}\right) \pmod2.
		\end{equation}
		Replacing $n$ by $pn$ we have,
		\begin{equation}\label{eq1114}
			a_2\left( p^4n+ \frac{3(p^4-1)}{8} \right) \equiv - p^5 \cdot a_2\left(n \right) \pmod2.
		\end{equation}
		Using mathematical induction in \eqref{eq1114}, we obtain for $n,k \geq 0,$
		\begin{equation}\label{eq1115}
			a_2\left( p^{4k}n+ \frac{3(p^{4k}-1)}{8} \right) \equiv (-p^5)^k \cdot a_2\left(n \right) \pmod2.
		\end{equation}
		Congruence \eqref{eq1113} implies that If $ w_2(p) \equiv 0 \pmod2$ and $ p \nmid n,$ then
		\begin{equation}\label{eq1116}
			a_2\left( p^3n+ \frac{3(p^4-1)}{8} \right) \equiv 0 \pmod2.
		\end{equation}
		Replacing $n$ by $ p^3n+ \frac{3(p^4-1)}{8}$ in \eqref{eq1115} and using \eqref{eq1116}, we get that if $ w_2(p) \equiv 0 \pmod2$ and $ p \nmid n,$ then
		\begin{equation}\label{eq1117}
			a_2\left( p^{4k+3}n+ \frac{3(p^{4k+4}-1)}{8} \right) \equiv 0 \pmod2.
		\end{equation}
		From \eqref{eq1106} and \eqref{eq1117} we obtain that, if $ w_2(p) \equiv 0 \pmod2$ and $ p \nmid n,$ then
		\begin{equation}\label{eq1117a}
			B_{8,3}\left( p^{4k+3}n+ \frac{3(p^{4k+4}-1)}{8} \right) \equiv 0 \pmod2.
		\end{equation}
		
		Replacing $n$ by  $ p^2n+ \frac{3p(p^2-1)}{8}$ in \eqref{eq1112}, we get
		\begin{equation}\label{eq1118}
			a_2\left( p^5n+ \frac{3(p^6-1)}{8} \right)=  w_2(p) a_2\left( p^3n+ \frac{3(p^4-1)}{8} \right)- p^5 \cdot a_2\left( pn+ \frac{3(p^2-1)}{8} \right).
		\end{equation}
		Combining \eqref{eq1118} and \eqref{eq1112} together, we get
		\begin{equation}\label{eq1119}
			a_2\left( p^5n+ \frac{3(p^6-1)}{8} \right)=  \left( w_2^2(p) -p^5  \right)  a_2\left( pn+ \frac{3(p^2-1)}{8} \right)-  w_2(p) \cdot p^5 \cdot  a_2\left( \frac{n}{p} \right).
		\end{equation}
		Note that if $ w_2(p) \not \equiv 0 \pmod2,$ then  $ w_2(p) \equiv 1 \pmod2$ and hence $ \left( w_2^2(p) -p^5  \right) \equiv 0 \pmod 2.$ Thus, if $ w_2(p) \not \equiv 0 \pmod2,$ then \eqref{eq1119} will be
		
		\begin{equation}\label{eq1120}
			a_2\left( p^5n+ \frac{3(p^6-1)}{8} \right) \equiv -  w_2(p) \cdot p^5 \cdot  a_2\left( \frac{n}{p} \right) \pmod2.
		\end{equation}
		Therefore, if $p \nmid n,$ then $a_2\left( \frac{n}{p} \right) =0$ and 
		\begin{equation}\label{eq1121}
			a_2\left( p^5n+ \frac{3(p^6-1)}{8} \right) \equiv 0 \pmod2.
		\end{equation}
		Replacing $n$ by $pn$ in \eqref{eq1120}, we obtain
		\begin{equation}\label{eq1122}
			a_2\left( p^6n+ \frac{3(p^6-1)}{8} \right) \equiv -  w_2(p) \cdot p^5 \cdot  a_2\left( n \right) \pmod2.
		\end{equation}
		By using mathematical induction in \eqref{eq1122}, we get
		\begin{equation}\label{eq1123}
			a_2\left( p^{6k}n+ \frac{3(p^{6k}-1)}{8} \right) \equiv  ( - w_2(p) \cdot p^5)^k \cdot  a_2\left( n \right) \pmod2.
		\end{equation}
		Replacing $n$ by $ p^5n+ \frac{3(p^6-1)}{8}$ in \eqref{eq1123} and using \eqref{eq1121}, it follows that if $ w_2(p) \not \equiv 0 \pmod2$ and $p \nmid n,$ we get 
		\begin{equation}\label{eq1124}
			a_2\left( p^{6k+5}n+ \frac{3(p^{6k+6}-1)}{8} \right) \equiv  0 \pmod2.
		\end{equation}
		From \eqref{eq1124} and \eqref{eq1106} we obtain that, if $ w_2(p) \not \equiv 0 \pmod2$ and $p \nmid n,$  then
		\begin{equation}\label{eq1125}
			B_{8,3}\left( p^{6k+5}n+ \frac{3(p^{6k+6}-1)}{8} \right) \equiv  0 \pmod2.
		\end{equation}
		Note that $\left( \frac{-6}{p} \right)_{L}  \left( \frac{n- \frac{3(p^2-1)}{8}}{p} \right)_{L} = \left( \frac{-6n-2 + \frac{(p^2-1)}{4}}{p} \right)_{L}. $
		From \eqref{eq1111} it follows that if $ w_2(p) \equiv  p^2 \cdot \left( \frac{-6n-2 + \frac{(p^2-1)}{4}}{p} \right)_{L} \pmod2 $ we get
		\begin{equation}\label{eq1126}
			a_2\left( p^2n+ \frac{3(p^2-1)}{8} \right) \equiv - p^5 \cdot a_2\left( \frac{n- \frac{3(p^2-1)}{8}}{p^2}\right) \pmod2.
		\end{equation}
		Observe that $ \left( \frac{-6n-2 + \frac{(p^2-1)}{4}}{p} \right)_{L}  \not \equiv 0 \pmod2,$ then $p \nmid (8n+3)$ and hence $ \left( \frac{n- \frac{3(p^2-1)}{8}}{p}\right) $ and $\left( \frac{n- \frac{3(p^2-1)}{8}}{p^2}\right) $ are not integers. Thus, for $n \geq 0,$ from \eqref{eq1126} 
		we have
		\begin{equation}\label{eq1127}
			a_2\left( p^2n+ \frac{3(p^2-1)}{8} \right) \equiv 0 \pmod2.
		\end{equation}
		Replacing $n$ by $p^2n+ \frac{3(p^2-1)}{8} $ in \eqref{eq1123} and using \eqref{eq1127}, we obtain that if  $w_2(p) \not \equiv 0 \pmod2, $ and $ w_2(p) \equiv  p^2 \cdot \left( \frac{-6n-2 + \frac{(p^2-1)}{4}}{p} \right)_{L} \pmod2,$ then
		
		\begin{equation}\label{eq1128}
			a_2\left( p^{6k+2}n+ \frac{3(p^{6k+2}-1)}{8} \right) \equiv 0 \pmod2.
		\end{equation}
		Combining \eqref{eq1106} and \eqref{eq1128} together, we get
		\begin{equation}\label{eq1129}
			B_{8,3}\left( p^{6k+2}n+ \frac{3(p^{6k+2}-1)}{8} \right) \equiv 0 \pmod2.
		\end{equation}
\end{proof}
	
	%%%%%%%%%%%%%%%%%%%%%%%%%%%%%%%%%%%%%%%%%%%%%%%%%%%%%%%%%%%%%%%%%%%%%%%%%%%
	%%%%%%%%%%%Newman Identity II
	%%%%%%%%%%%%%%%%%%%%%%%%%%%%%%%%%%%%%%%%%%%%%%%%%%%%%%%%%%%%%%%%%%%%%%%%%%%%%%%%
	%%%%%%%%%%%%%%%%%%%%%%%%%%%%%%%%%%%%%%%%%%%%%%%%%%%%%%%%%%%%%%%%%%%%%%%%%%%
	\section{Congruences for $B_{4,5}(n)$}
	%%%%%%%%%%%%%%%%%%%%%%%%%%%%%%%%%%%%%%%%%%%%%%%%%%%%%%%%%%%%%%%%%%%%%%%%%%%%%%%%

\begin{proof}[{\bf Proof of Theorem \ref{thmnewman12}}]
		Note that 
		\begin{equation}\label{eq1206}
			\sum_{n=0}^{\infty} B_{4,5}(n)q^n=  \frac{(q^4;q^4)_{\infty} (q^5;q^5)_{\infty}}{(q;q)^2_{\infty}} \equiv (q;q)^2_{\infty} (q^5;q^5)_{\infty} \pmod2.
		\end{equation} 
		Let $a_3(n)$ be defined by \eqref{eq1201}. Substituting $q=5, r=2, s=1$ in Lemma \ref{lem100}, we obtain
		\begin{equation}\label{eq1207}
			a_3\left( p^2n+ \frac{7(p^2-1)}{24}\right)= \gamma_{3}(n) a_3(n)- p \cdot a_3\left( \frac{n- \frac{7(p^2-1)}{24}}{p^2}\right),
		\end{equation}
		where 
		\begin{equation}\label{eq1208}
			\gamma_{3}(n)= p \cdot c - \left( \frac{-10}{p} \right)_{L}  \left( \frac{n- \frac{7(p^2-1)}{24}}{p} \right)_{L}.
		\end{equation}
		If we put $n=0$ in \eqref{eq1207} and using the facts that $a_3(0)=1$ and $a_3\left( \frac{- \frac{7(p^2-1)}{24}}{p^2}\right) =0, $ we get
		\begin{equation}\label{eq1209}
			\gamma_{3}(0)= a_3\left( \frac{7(p^2-1)}{24}\right).
		\end{equation}
		Further taking $n=0$ in \eqref{eq1208} and using \eqref{eq1209}, we obtain
		\begin{equation}\label{eq1210}
			pc= w_3(p),
		\end{equation}
		where $ w_3(p) :=  a_3\left( \frac{7(p^2-1)}{24}\right) + \left( \frac{-10}{p}\right)_{L} \left( \frac{-7(p^2-1)}{24 p}\right)_{L}.$
		Combining \eqref{eq1210}, \eqref{eq1207} and \eqref{eq1208} together, we get
		\begin{equation}\label{eq1211}
			a_3\left( p^2n+ \frac{7(p^2-1)}{24} \right)= \left[ w_3(p) - \left( \frac{-10}{p} \right)_{L}  \left( \frac{n- \frac{7(p^2-1)}{24}}{p} \right)_{L} \right] a_3(n)- p \cdot a_3\left( \frac{n- \frac{7(p^2-1)}{24}}{p^2}\right).
		\end{equation}
		Replacing $n$ by $ pn+ \frac{7(p^2-1)}{24} $ in \eqref{eq1211}, we obtain
		\begin{equation}\label{eq1212}
			a_3\left( p^3n+ \frac{7(p^4-1)}{24} \right)=  w_3(p) a_3\left( pn+ \frac{7(p^2-1)}{24} \right)- p \cdot a_3\left( \frac{n}{p}\right).
		\end{equation}
		If $ w_3(p) \equiv 0 \pmod2,$ then 
		\begin{equation}\label{eq1213}
			a_3\left( p^3n+ \frac{7(p^4-1)}{24} \right) \equiv - p \cdot a_3\left( \frac{n}{p}\right) \pmod2.
		\end{equation}
		Replacing $n$ by $pn$ we have,
		\begin{equation}\label{eq1214}
			a_3\left( p^4n+ \frac{7(p^4-1)}{24} \right) \equiv - p \cdot a_3\left(n \right) \pmod2.
		\end{equation}
		Using mathematical induction in \eqref{eq1214}, we obtain for $n,k \geq 0,$
		\begin{equation}\label{eq1215}
			a_3\left( p^{4k}n+ \frac{7(p^{4k}-1)}{24} \right) \equiv (-p)^k \cdot a_3\left(n \right) \pmod2.
		\end{equation}
		Congruence \eqref{eq1213} implies that If $ w_3(p) \equiv 0 \pmod2$ and $ p \nmid n,$ then
		\begin{equation}\label{eq1216}
			a_3\left( p^3n+ \frac{7(p^4-1)}{24} \right) \equiv 0 \pmod2.
		\end{equation}
		Replacing $n$ by $ p^3n+ \frac{7(p^4-1)}{24}$ in \eqref{eq1215} and using \eqref{eq1216}, we get that if $ w_2(p) \equiv 0 \pmod2$ and $ p \nmid n,$ then
		\begin{equation}\label{eq1217}
			a_3\left( p^{4k+3}n+ \frac{7(p^{4k+4}-1)}{24} \right) \equiv 0 \pmod2.
		\end{equation}
		From \eqref{eq1206} and \eqref{eq1217} we obtain that, if $ w_2(p) \equiv 0 \pmod2$ and $ p \nmid n,$ then
		\begin{equation}\label{eq1217a}
			B_{4,5}\left( p^{4k+3}n+ \frac{7(p^{4k+4}-1)}{24} \right) \equiv 0 \pmod2.
		\end{equation}
		Replacing $n$ by  $ p^2n+ \frac{7p(p^2-1)}{24}$ in \eqref{eq1212}, we get
		\begin{equation}\label{eq1218}
			a_3\left( p^5n+ \frac{7(p^6-1)}{24} \right)=  w_3(p) a_3\left( p^3n+ \frac{7(p^4-1)}{24} \right)- p \cdot a_3\left( pn+ \frac{7(p^2-1)}{24} \right).
		\end{equation}
		Combining \eqref{eq1218} and \eqref{eq1212} together, we get
		\begin{equation}\label{eq1219}
			a_3\left( p^5n+ \frac{7(p^6-1)}{24} \right)=  \left( w_3^2(p) -p  \right)  a_3\left( pn+ \frac{7(p^2-1)}{24} \right)-  w_3(p) \cdot p \cdot  a_3\left( \frac{n}{p} \right).
		\end{equation}
		Note that if $ w_3(p) \not \equiv 0 \pmod2,$ then  $ w_3(p) \equiv 1 \pmod2$ and hence $ \left( w_3^2(p) -p  \right) \equiv 0 \pmod 2.$ Thus, if $ w_3(p) \not \equiv 0 \pmod2,$ then \eqref{eq1219} will be
		
		\begin{equation}\label{eq1220}
			a_3\left( p^5n+ \frac{7(p^6-1)}{24} \right) \equiv -  w_3(p) \cdot p \cdot  a_3\left( \frac{n}{p} \right) \pmod2.
		\end{equation}
		Therefore, if $p \nmid n,$ then $a_3\left( \frac{n}{p} \right) =0$ and 
		\begin{equation}\label{eq1221}
			a_3\left( p^5n+ \frac{7(p^6-1)}{24} \right) \equiv 0 \pmod2.
		\end{equation}
		Replacing $n$ by $pn$ in \eqref{eq1220}, we obtain
		\begin{equation}\label{eq1222}
			a_3\left( p^6n+ \frac{7(p^6-1)}{24} \right) \equiv -  w_3(p) \cdot p \cdot  a_3\left( n \right) \pmod2.
		\end{equation}
		By using mathematical induction in \eqref{eq1222}, we get
		\begin{equation}\label{eq1223}
			a_3\left( p^{6k}n+ \frac{7(p^{6k}-1)}{24} \right) \equiv  ( - w_3(p) \cdot p)^k \cdot  a_3\left( n \right) \pmod2.
		\end{equation}
		Replacing $n$ by $ p^5n+ \frac{7(p^6-1)}{24}$ in \eqref{eq1223} and using \eqref{eq1221}, it follows that if $ w_3(p) \not \equiv 0 \pmod2$ and $p \nmid n,$ we get 
		\begin{equation}\label{eq1224}
			a_3\left( p^{6k+5}n+ \frac{7(p^{6k+6}-1)}{24} \right) \equiv  0 \pmod2.
		\end{equation}
		From \eqref{eq1224} and \eqref{eq1206} we obtain that, if $ w_3(p) \not \equiv 0 \pmod2$ and $p \nmid n,$  then
		\begin{equation}\label{eq1225}
			B_{4,5}\left( p^{6k+5}n+ \frac{7(p^{6k+6}-1)}{24} \right) \equiv  0 \pmod2.
		\end{equation}
		Note that $\left( \frac{-10}{p} \right)_{L}  \left( \frac{n- \frac{7(p^2-1)}{24}}{p} \right)_{L} = \left( \frac{-10n-2 + \frac{11(p^2-1)}{12}}{p} \right)_{L}. $
		From \eqref{eq1211} it follows that if $ w_3(p) \equiv \left( \frac{-10n-2 + \frac{11(p^2-1)}{12}}{p} \right)_{L} \pmod2 $ we get
		\begin{equation}\label{eq1226}
			a_3\left( p^2n+ \frac{7(p^2-1)}{24} \right) \equiv - p \cdot a_3\left( \frac{n- \frac{7(p^2-1)}{24}}{p^2}\right) \pmod2.
		\end{equation}
		Observe that $ \left( \frac{-10n-2 + \frac{11(p^2-1)}{12}}{p} \right)_{L}  \not \equiv 0 \pmod2,$ then $p \nmid (24n+7)$ and hence $ \left( \frac{n- \frac{7(p^2-1)}{24}}{p}\right) $ and $\left( \frac{n- \frac{7(p^2-1)}{24}}{p^2}\right) $ are not integers. Thus, for $n \geq 0,$ from \eqref{eq1226} 
		we have
		\begin{equation}\label{eq1227}
			a_3\left( p^2n+ \frac{7(p^2-1)}{24} \right) \equiv 0 \pmod2.
		\end{equation}
		Replacing $n$ by $p^2n+ \frac{7(p^2-1)}{24} $ in \eqref{eq1223} and using \eqref{eq1227}, we obtain that if  $w_3(p) \not \equiv 0 \pmod2, $ and $ w_3(p) \equiv \left( \frac{-10n-2 + \frac{11(p^2-1)}{12}}{p} \right)_{L} \pmod2,$ then
		\begin{equation}\label{eq1228}
			a_3\left( p^{6k+2}n+ \frac{7(p^{6k+2}-1)}{24} \right) \equiv 0 \pmod2.
		\end{equation}
		Combining \eqref{eq1206} and \eqref{eq1228} together, we get
		\begin{equation}\label{eq1229}
			B_{4,5}\left( p^{6k+2}n+ \frac{7(p^{6k+2}-1)}{24} \right) \equiv 0 \pmod2.
		\end{equation}
	\end{proof}
	
	\noindent{\bf Data availability statement:} There is no data associated to our manuscript.
	%	
	%	\noindent{\bf Conflict of interest:} There is no conflict of interest between the two authors.
	
	%	\noindent{\bf Acknowledgement.}

\end{document}